\documentclass[12pt,reqno]{amsart}
\usepackage{amssymb}
\usepackage{amsmath}
\usepackage{graphicx}
\usepackage{color}
\usepackage{epsf}
\usepackage{psfrag}
\usepackage{epsfig}
\usepackage{comment}
\usepackage{a4wide}
\usepackage{longtable}

\newtheorem{theorem}{Theorem}
\newtheorem{corollary}[theorem]{Corollary}
\newtheorem{lemma}[theorem]{Lemma}
\newtheorem{proposition}[theorem]{Proposition}

\def\Z{\mathbb Z}

\def\bei{\begin{itemize}}\def\eni{\end{itemize}}
\def\bee{\begin{enumerate}}\def\ene{\end{enumerate}}

\begin{document}

\title[Small integer symmetric matrices]{Integer symmetric matrices of small\\ spectral radius and small Mahler measure}
\author{ James McKee}
\address{Department of Mathematics\\
Royal Holloway, University of London\\
Egham Hill\\
Egham\\
Surrey TW20 0EX\\
UK}
\email{James.McKee@rhul.ac.uk}
\author{ Chris Smyth}
\address{School of Mathematics and Maxwell Institute for Mathematical Sciences\\
University of Edinburgh\\
James Clerk Maxwell Building\\
King's Buildings, Mayfield Road\\
Edinburgh EH9 3JZ\\
Scotland, U.K.} \email{C.Smyth@ed.ac.uk} 
\date{}

\begin{abstract}
In a previous paper we completely described \emph{cyclotomic}
matrices---integer symmetric matrices of spectral radius at most
$2$. In this paper we find all minimal noncyclotomic matrices. As
a consequence, we are able to determine all integer symmetric
matrices of spectral radius at most $2.019$, and to determine all
integer symmetric matrices whose Mahler measure is at most $1.3$.
In particular we solve the strong version of Lehmer's problem for
integer symmetric matrices: all noncyclotomic matrices have Mahler
measure at least `Lehmer's number' $\lambda_0=1.17628\dots$\,\,.
\end{abstract}

\maketitle

\section{Statement of results}

For  a monic polynomial $g(x)$ with integer coefficients and
degree $d$, define $z^d g(z+1/z)$, a monic reciprocal polynomial
of degree $2d$, to be its \emph{associated reciprocal polynomial.}
If $g(z)$ has all its roots real and in the interval $[-2,2]$,
then the roots of its associated reciprocal polynomial are all of
modulus 1, and, by a theorem of Kronecker \cite{Kro},  it is a
cyclotomic polynomial. If $A$ is a $d$-by-$d$ symmetric matrix
with integer entries, then all the roots of its characteristic
polynomial are real algebraic integers, and we denote by
$R_{A}(z)$ its associated reciprocal polynomial.
If $A$ has spectral radius at most $2$, so that $R_A(z)$ is
cyclotomic, then we say that $A$ is a \emph{cyclotomic matrix}. In
a previous paper \cite{MScyc} we completely described cyclotomic
matrices.

We are interested in the spectrum of values taken by two different
functions of integer symmetric matrices $A$. The first is the
spectral radius. In this paper we find all integer symmetric
matrices whose spectral radius is less than $2.019$.

The second function we are interested in is the
 \emph{Mahler measure} of $A$, defined to be the Mahler measure
of $R_A(z)$, namely the product of the absolute values of all roots
of $R_A(z)$ that have modulus greater than 1. The Mahler measure of
$A$ equals 1 precisely when $A$ is a cyclotomic matrix.

Is there a (monic) polynomial with integer coefficients that has
Mahler measure $\lambda$ greater than 1 but such that every
polynomial with integer coefficients has Mahler measure either 1
or at least $\lambda$? This is Lehmer's famous problem. The
smallest known such Mahler measure greater than 1 is
\begin{equation}\label{eq:lam_0}
\lambda_0 = 1.176280818\ldots\,,
\end{equation}
the larger real root of the polynomial
\begin{equation*}\label{eq:pol}
z^{10} + z^9 - z^7 - z^6 - z^5 - z^4 - z^3 + z + 1\,.
\end{equation*}
Is there a Mahler measure in the open interval $(1,\lambda_0)$?

Restricting to certain classes of polynomials, the analogue of
Lehmer's problem may be more accessible.  For example, Borwein,
Dobrowolski and Mossinghoff \cite{BDM} settled the problem for
those polynomials that have only odd coefficients (the Mahler
measure is either $1$ or at least $1.4953\ldots$), and Dobrowolski
\cite{D} showed that if $A$ is an integer symmetric matrix, then
$R_A(z)$ has Mahler measure either 1 or at least
\[
\lambda_1 = 1.043\ldots\,\,.
\]

In this paper (Corollary \ref{C:main}), we use a completely
different approach to strengthen Dobrowolski's result and to show
that the Mahler measure of an integer symmetric matrix is either 1
or at least $\lambda_0$ given by (\ref{eq:lam_0}). This constant
is best possible here, as there are integer symmetric matrices
that have Mahler measure equal to $\lambda_0$, the smallest being
 \[ \left(
\begin{array}{ccccc}
0 & 1 & 0 & 0 & 0 \\
1 & 0 & 1 & 0 & 0 \\
0 & 1 & 0 & 1 & 1 \\
0 & 0 & 1 & 0 & 1 \\
0 & 0 & 1 & 1 & -1
\end{array}
\right) {\rm \ \ and\ \ } \left(
\begin{array}{ccccc}
1 & 1 & 0 & 0 & 0 \\
1 & -1 & 1 & 0 & 0 \\
0 & 1 & 0 & 1 & 0 \\
0 & 0 & 1 & 0 & 1 \\
0 & 0 & 0 & 1 & 1
\end{array}
\right)\,.
\]

Actually, these are the adjacency matrices of the charged signed
graphs $5u$ and $5N$ of Section \ref{S:minnon}.  In fact we go
further: in Theorem \ref{T:measure} we find all integer symmetric
matrices having Mahler measure less than $1.3$.

Adjacency matrices of graphs form a proper subset of the set of
all integer symmetric matrices: the entries are either 0 or 1,
with zeros on the diagonal. For these special matrices, the
analogue of Lehmer's problem was  solved in \cite[Corollary
10.1]{MS}, based on work of  Brouwer and Neumaier \cite{BN} and
Cvetkovi\'{c}, Doob, and Gutman \cite{CDG}. A key ingredient in
the current paper is an extension of this work on graphs to
cover signed graphs (allowing $-1$ as an off-diagonal matrix
entry) and charged signed graphs (allowing $\pm 1$ entries on the
diagonal). Matrices having any entry of modulus $2$ or more are
easily disposed of.

We now state our main theorem, the one from which our results on
the spectrum of the spectral radius and Mahler measure of integer
symmetric matrices are derived. The terms used in its statement
are defined in the next section.

\begin{theorem}\label{T:main}
Up to equivalence, the     minimal noncyclotomic integer symmetric
matrices are those catalogued below in Section \ref{S:minnon}. There
are $4$ infinite families and $125$ sporadic examples.
\end{theorem}

This result generalises a theorem of Cvetkovi\'{c}, Doob, and
Gutman \cite{CDG}, who found all $18$ minimal noncyclotomic graphs
(i.e., restricting to integer symmetric matrices with $\{0,1\}$
entries and zeros on the diagonal). See also \cite[Theorem
2.3]{CR}. These $18$ graphs are precisely the graphs (i.e., the
charged signed graphs without charges or negatively signed edges)
that appear in our catalogue: $4a$, $4j$, $4B$, $5a$, $5E$, $6c$,
$6d$, $6h$, $7a$, $7b$, $8a$, $8b$, $8c$, $9a$, $9b$, $9d$, $10a$,
$10f$.

By an interlacing argument, we immediately obtain the following
corollary.

\begin{corollary}\label{C:main}
If $A$ is an integer symmetric matrix, then the Mahler measure of
$A$ is either $1$ or at least $\lambda_0$ given by
(\ref{eq:lam_0}).

Furthermore, if $A$ is indecomposable and has Mahler measure equal
to $\lambda_0$ then it is equivalent to the adjacency matrix of
one of the charged signed graphs $5u$, $5N$ or $10f$.
\end{corollary}

For the spectral radius, this implies that if an integer symmetric matrix is
not cyclotomic then its spectral radius is at least
$\sqrt{\lambda_0}+1/\sqrt{\lambda_0}=2.00659\dots$\,\,. However,
we can build on Theorem \ref{T:main} to obtain the following
stronger results.

\begin{theorem}\label{T:spectral radius}
Up to equivalence, the indecomposable integer symmetric matrices
having spectral radius less than $2.019$ are either cyclotomic or
have spectral radius equal to one of the ten values given in Table
\ref{Ta-1}. The matrices having each such spectral radius are also
given in this table.
\end{theorem}

\begin{table}[ht]
\begin{center}
\bigskip\begin{tabular} { c | c | c}\label{Ta:spectral radius}

\# &Maximum modulus  & Charged signed graph having\\
 & of eigenvalues & corresponding spectral radius \\
     \hline
1 & 2.00659 & $10f=T_{1,2,6}$    \\
2 & 2.00960 & $10e$\\
3 & 2.01076 & $11c = T_{1,2,7}$\\
4 & 2.01348 & $10d$, $12b = T_{1,2,8}$\\
5 & 2.01532 & $9d = T_{1,3,4}$, $10g$, $11a$, $11b$, $13a = T_{1,2,9}$\\
6 & 2.01658 & $10h$, $14a = T_{1,2,10}$\\
7 & 2.01746 & $15a = T_{1,2,11}$\\
8 & 2.01809 & $16a = T_{1,2,12}$\\
9 & 2.01854 & $17a = T_{1,2,13}$\\
10 & 2.01887 & $12a$, $18a = T_{1,2,14}$\\

   \end{tabular}
    \vspace*{1ex}
    \caption{The noncyclotomic connected charged signed graphs whose eigenvalues are at most
    $2.019$ in modulus. The graphs are drawn in Sections \ref{S:minnon} and \ref{S:SSR}.} \label{Ta-1}
\end{center}\end{table}

\begin{picture}(300,180)(0,-20)
\put(90,0){Figure 1: The tree $T_{a,b,c}$.}

\multiput(70,30)(0,40){3}{\circle*{10}}
\multiput(100,30)(0,40){3}{\circle*{10}}
\multiput(160,30)(0,40){3}{\circle*{10}}
\put(220,70){\circle*{10}}
\multiput(70,30)(0,40){3}{\line(1,0){45}}
\multiput(145,30)(0,40){3}{\line(1,0){15}}
\multiput(125,27.5)(0,40){3}{$\cdots$}
\put(160,70){\line(1,0){60}} \put(160,30){\line(3,2){60}}
\put(160,110){\line(3,-2){60}}
\multiput(70,40)(0,40){3}{$\overbrace{\hspace{90pt}}$}
\put(112,50){$c$} \put(112,90){$b$} \put(112,130){$a$}
\end{picture}

The choice of $2.019$ for our bound was governed by the fact that
as $n\to\infty$ the spectral radius of the graph $T_{1,2,n}$
(Figure $1$) tends to
$\sqrt{\theta_0}+1/\sqrt{\theta_0}=2.019800887\dots$\,. Here
$\theta_0$ is the real root of $x^3-x-1$. Thus if our upper bound
were to be above $2.019800887\dots$ it would include infinitely
many integer symmetric matrices. For graphs  this was done by
Cvetkovi\'{c}, Doob, and Gutman \cite{CDG} and Brouwer and
Neumaier \cite{BN}, who in fact found all graphs of spectral
radius less than $\sqrt{2+\sqrt{5}}=2.058171027\dots$~.  See also
\cite[Theorem 2.4]{CR}. Furthermore Shearer \cite{She} showed that
the set of all spectral radii of graphs is dense in
$(\sqrt{2+\sqrt{5}},\infty)$. It would be nice to extend the
analysis for $(2,\sqrt{2+\sqrt{5}})$ to general integer symmetric
matrices.

For the Mahler measure, we have a corresponding result.

\begin{theorem}\label{T:measure}
Up to equivalence, the indecomposable integer symmetric matrices
having Mahler measure less than $1.3$ are either cyclotomic or
have Mahler measure equal to one of the sixteen values given in
Table \ref{Ta-2}. The matrices having each such Mahler measure are
also given in this table.
\end{theorem}

\begin{table}[ht]
\begin{center}
\bigskip\begin{tabular} { c | c | c}\label{Ta:measure}

\# &Mahler measure  & Charged signed graph having\\
 & & corresponding Mahler measure \\
      \hline
1 & 1.17628 & $10f = T_{1,2,6}$, $5N$, $5u$   \\
2 & 1.18837 & $9e$\\
3 & 1.20003 & $8d$\\
4 & 1.21639 & $5M$, $5p$, $7c$, $10e$\\
5 & 1.21972 & $9g$\\
6 & 1.23039 & $11c = T_{1,2,7}$, $5L$, $6v$\\
7 & 1.23632 & $8e$\\
8 & 1.24073 & $6m$\\
9 & 1.25364 & $10c$\\
10 & 1.25622 & $9h$\\
11 & 1.26123 & $5F$, $5K$, $6t$, $6u$, $7d$, $10d$, $12b = T_{1,2,8}$\\
12 & 1.26730 & $8h$, $9i$, $10i$\\
13 & 1.28064 & $4I$, $5o$, $5x$, $6q$, $6r$, $6s$, $8f$, $9d$, $10g$, $11a$, $11b$, $13a = T_{1,2,9}$\\
14 & 1.28929 & $11d$\\
15 & 1.29349 & $5J$, $5O$, $7e$, $7f$, $7g$, $8g$, $9j$, $10h$, $14a = T_{1,2,10}$\\
16 & 1.29568 & $9f$\\
   \end{tabular}
    \vspace*{1ex}
    \caption{The noncyclotomic connected charged signed graphs whose Mahler measure is less than $1.3$. The graphs are drawn in Sections \ref{S:minnon} and \ref{S:SMM}.}
    \label{Ta-2}
\end{center}\end{table}

Concerning the choice of $1.3$ for the bound, we first note that
for a finite list of matrices we must have the bound less than
$\theta_0=1.324717957\dots$, the limit as $n\to\infty$ of the
Mahler measures of the sequence of graphs $T_{1,2,n}$. Also, Boyd
in his 1977 table of small Salem numbers \cite{Bo77} (later
slightly extended both by Boyd himself and by Mossinghoff) used
the bound $1.3$. See also \cite{S}. And all but two of the
reciprocal polynomials of the matrices in Table \ref{Ta-2} are
(apart from  possible cyclotomic factors) minimal polynomials of
Salem numbers: the exceptions are the (reciprocal polynomials of
the) charged signed graphs $10c$ and $11d$.  For the extended
table see \cite{Moss}. Of the $47$ Salem numbers, our result shows
that only $14$ of them are Mahler measures of integer symmetric
matrices.

In \cite{MS} we in fact showed that  from  \cite{BN} and
\cite{CDG} mentioned above one could describe all graphs having
Mahler measure less than
$\frac{1}{2}(1+\sqrt{5})=1.61803\dots$\,\,. Again, it would be
nice to improve Theorem \ref{T:measure}  by increasing the bound
$1.3$ up to this value.

In \cite{D} it was shown, contrary to a conjecture of Estes and
Guralnick \cite{EG}, that there are infinitely many totally real
algebraic integers whose minimal polynomial is not the
characteristic polynomial of an integer symmetric matrix.  The
counterexamples in \cite{D} were all cyclotomic: the conjugates of
the real algebraic integers being all in the interval $[-2,2]$.
Comparing our Table \ref{Ta-2} with the tables of small Salem
numbers in \cite{Bo77} and \cite{Moss}, we find some noncyclotomic
counterexamples to the conjecture of Estes and Guralnick: for
example, if $x^7-8x^5+19x^3-12x+1$ were the characteristic
polynomial of an integer symmetric matrix $A$, then $R_A(z)$ would
be $z^{14}-z^{12}+z^7-z^2+1$, and the Mahler measure of $A$ would
be $1.20261\dots$, but this does not appear in Table \ref{Ta-2}.

\section{Some definitions}
A \emph{permutation} of a $d$-by-$d$ matrix $A=(a_{ij})$ is a
matrix $P^TAP$, where $P$ is a $d$-by-$d$ permutation matrix.
More generally, a \emph{signed permutation} of $A=(a_{ij})$ is a matrix $P^TAP$, where $P$ is a
$d$-by-$d$ \emph{signed permutation matrix}, that is an element of
the group $O_d(\Z)$ of orthogonal matrices with integer entries
(which must be $0$ or $\pm 1$). If such a $P$ is a diagonal
matrix, we obtain a \emph{switching} of $A$. In other words, to
perform a switching we change the signs of some subset of the
rows, and then change the signs of the same subset of the columns.
Note that diagonal entries are unchanged by a switching. In
particular, {\it switching a vertex} means changing the signs of
all edges incident at that vertex.

Two $d$-by-$d$ integer symmetric matrices $A$ and $B$ will be
called \emph{equivalent} if $B=\pm P^TAP$ for some $P\in O_d(\Z)$.
Since a signed permutation matrix is a product of a permutation
matrix and a diagonal matrix, we see then that  one can be
transformed to the other by a permutation and a switching, and
then possibly a change of sign of all the entries. Equivalent
matrices have the same spectral radius and the same Mahler measure. Applying
a permutation or a switching does not change the eigenvalues;
changing the signs of all entries changes the signs of all
eigenvalues. If the matrices are restricted to being the adjacency
matrices of signed graphs, then our equivalence classes almost
correspond to the `signed switching classes' of Cameron et
al.~\cite{Cetal}, except that we allow a change of sign of all the
edges: our class is the union of one or two signed switching
classes.

Any $d$-by-$d$ integer symmetric matrix $A=(a_{ij})$ can be viewed
as the adjacency matrix of a signed graph with charges: take $d$
vertices, and for $i\ne j$ put $|a_{ij}|$  edges with the same
sign as $a_{ij}$ joining vertex $i$ and vertex $j$. At the $i$th
vertex place a charge of $a_{ii}$. The matrix $A$ is called
\emph{indecomposable} if this graph is connected. If $A$ is not
indecomposable, then it is \emph{decomposable}: this implies that
some permutation of $A$ is in block-diagonal form, with more than
one block, and the eigenvalues of $A$ are obtained by pooling the
eigenvalues of its blocks. The spectral radius of $A$ is clearly
the maximum spectral radius of its blocks, while the Mahler
measure of $A$ is the product of the Mahler measures of its
blocks. For our purposes it is therefore sufficient to consider
indecomposable matrices.

Whenever we refer to a \emph{subgraph}, we shall mean one induced
by a subset of its vertices.  A \emph{submatrix} of a square
matrix $A$ is one that is obtained from $A$ by deleting a subset
of the rows, and deleting the same subset of the columns. With the
graphical interpretation, a submatrix is the adjacency matrix of a
subgraph.

Complementary to our definition of a cyclotomic matrix, we say
that an integer symmetric matrix is \emph{noncyclotomic} if its
spectral radius is greater than $2$.  A $d$-by-$d$ integer
symmetric matrix is \emph{minimal noncyclotomic} if it is
noncyclotomic and every $(d-1)$-by-$(d-1)$ submatrix is
cyclotomic. A minimal noncyclotomic matrix is necessarily
indecomposable.

We shall need the interlacing theorem of Cauchy (see
\cite[Th\'eor\`eme I, p.187]{Ca}, quoted in \cite[pp.59,
78]{Bh96}; see also \cite{F} for a short proof):
\begin{lemma}
Let $A$ be a $d$-by-$d$ integer symmetric matrix, with eigenvalues
$\lambda_1 \le \lambda_2 \le \cdots \le \lambda_d$. If $B$ is any
$(d-1)$-by-$(d-1)$ principal submatrix of $A$, with eigenvalues
$\mu_1\le \mu_2\le\cdots\le\mu_{d-1}$, then the eigenvalues of $A$
and $B$ interlace:
\[
\lambda_1 \le \mu_1 \le \lambda_2 \le \mu_2 \le \cdots
\le \lambda_{d-1}\le\mu_{d-1}\le\lambda_d\,.
\]
\end{lemma}

By interlacing, any noncyclotomic integer symmetric matrix $A$
contains a minimal noncyclotomic submatrix $M$ (perhaps more than
one). Moreover, interlacing shows that the Mahler measure of $A$
is at least that of $M$. We see that Corollary \ref{C:main}
follows from Theorem \ref{T:main}.

All of the terms in the statement of Theorem \ref{T:main} have now
been defined. To prove the theorem, we shall need a good
description of all \emph{maximal cyclotomic} indecomposable
integer symmetric matrices $A$, i.e., those indecomposable
cyclotomic integer symmetric $A$ such that if $A$ is a submatrix
of $B$ and $B$ is cyclotomic and indecomposable then $A=B$.

\section{Maximal cyclotomic integer symmetric
matrices}\label{S:maxcyc}

We recall from \cite{MScyc} the maximal indecomposable cyclotomic
integer symmetric matrices (up to equivalence). There are seven
sporadic examples, $S_1$, $S_2$, $S_7$, $S_8$, $S_8'$, $S_{14}$,
$S_{16}$, and three infinite families $T_{2k}$, $C_{2k}^{++}$,
$C_{2k}^{+-}$.

The first two sporadic examples are
\[
S_1 = (2)\,,
\]
\[
S_2=\left(\begin{array}{cc} 0 & 2 \\ 2 & 0 \end{array}\right).
\]
The remaining examples have all entries in $\{-1,0,1\}$, so are
conveniently (and more compactly) represented by charged signed
graphs, with the following conventions: (i) positive edges are
represented by solid lines -----------; (ii) negative edges are
represented by dotted lines
\begin{picture}(30,10)\multiput(5,5)(5,0){5}{\circle*{2}}\end{picture};
(iii) neutral vertices ($0$ on the leading diagonal) are
represented by a solid disc
\begin{picture}(15,10)\put(5,3){\circle*{10}}\end{picture};
(iv) vertices with a positive charge ($+1$ on the   diagonal)
are represented by \textcircled{+}; (v) vertices with a negative
charge ($-1$ on the diagonal) are represented by \textcircled{--}.
They are shown in the pictures below, taken from \cite{MScyc}.

\begin{picture}(310,105)

\put(0,95){The three sporadic maximal cyclotomic charged signed
graphs:}

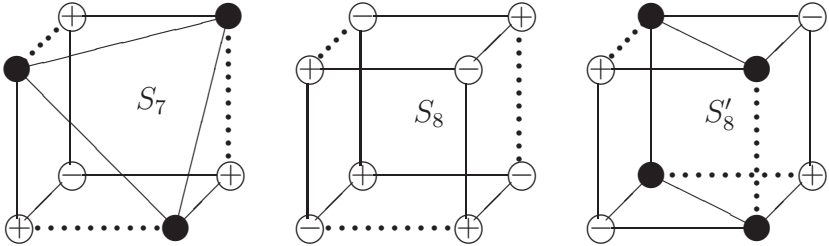
\put(-5,-3.5){\textcircled{+}} \put(0,60){\circle*{10}}
\put(15,16.5){\textcircled{$-$}} \put(15,76.5){\textcircled{+}}
\put(60,0){\circle*{10}} \put(75,16.5){\textcircled{+}}
\put(80,80){\circle*{10}}

\multiput(8,0)(5,0){10}{\circle*{2}} \put(3,3){\line(1,1){14}}
\put(0,3.5){\line(0,1){57}} \put(0,60){\line(1,-1){60}}
\put(0,60){\line(4,1){80}} \multiput(0,60)(3,3){6}{\circle*{2}}
\put(25,20){\line(1,0){50.5}} \put(20,24){\line(0,1){50.5}}
\put(25,80){\line(1,0){55}} \put(60,0){\line(1,1){17}}
\put(60,0){\line(1,4){20}} \multiput(80,27)(0,5){10}{\circle*{2}}

\put(45,45){$S_7$}

\put(105,-3.5){\textcircled{$-$}} \put(105,56.5){\textcircled{+}}
\put(125,16.5){\textcircled{+}} \put(125,76.5){\textcircled{$-$}}
\put(165,-3.5){\textcircled{+}} \put(165,56.5){\textcircled{$-$}}
\put(185,16.5){\textcircled{$-$}} \put(185,76.5){\textcircled{+}}

\put(150,40){$S_8$}

\multiput(118,0)(5,0){10}{\circle*{2}} \put(113,3){\line(1,1){14}}
\put(110,3.5){\line(0,1){51}} \put(115,60){\line(1,0){50.5}}
\multiput(115,65)(3,3){4}{\circle*{2}}
\put(135,20){\line(1,0){50.5}} \put(130,23.5){\line(0,1){51}}
\put(135,80){\line(1,0){50.5}} \put(173,3){\line(1,1){14}}
\put(170,3.5){\line(0,1){51}} \put(173,63){\line(1,1){14}}
\multiput(190,26)(0,5){10}{\circle*{2}}

\put(215,-3.5){\textcircled{$-$}} \put(215,56.5){\textcircled{+}}
\put(240,20){\circle*{10}} \put(240,80){\circle*{10}}
\put(280,0){\circle*{10}} \put(280,60){\circle*{10}}
\put(295,16.5){\textcircled{+}} \put(295,76.5){\textcircled{$-$}}

\put(260,40){$S_8'$}

\put(225,0){\line(1,0){55}} \put(223,3){\line(1,1){17}}
\put(220,3.5){\line(0,1){51}} \put(225,60){\line(1,0){55}}
\multiput(225,65)(3,3){5}{\circle*{2}}
\put(240,20){\line(2,-1){40}}
\multiput(248,20)(5,0){10}{\circle*{2}}
\put(240,20){\line(0,1){60}} \put(240,80){\line(2,-1){40}}
\put(240,80){\line(1,0){55.5}} \put(280,0){\line(1,1){17}}
\multiput(280,7)(0,5){10}{\circle*{2}}
\put(280,60){\line(1,1){17}} \put(300,23.5){\line(0,1){51}}

\end{picture}

\begin{picture}(220,250)(-70,0)

\put(-70,200){The $14$-vertex sporadic maximal cyclotomic signed
graph $S_{14}$:}

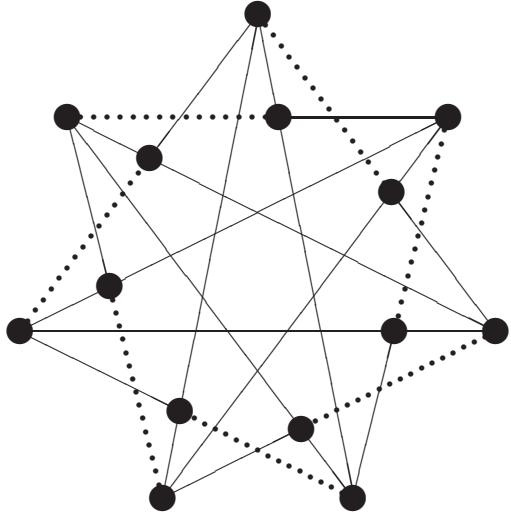
\put(6,63){\circle*{10}} \put(24,144){\circle*{10}}
\put(60,0){\circle*{10}} \put(96,183){\circle*{10}}
\put(132,0){\circle*{10}} \put(168,144){\circle*{10}}
\put(186,63){\circle*{10}}

\put(103.8,144){\circle*{10}} \put(147.75,63){\circle*{10}}
\put(146.625,115.5){\circle*{10}} \put(112.36,26.18){\circle*{10}}
\put(66.5,32.7){\circle*{10}} \put(40,80){\circle*{10}}
\put(55.1,128.4){\circle*{10}}

\put(6,63){\line(2,-1){60.5}} \put(6,63){\line(1,0){180}}
\put(6,63){\line(2,1){162}} \multiput(6,63)(3,4){16}{\circle*{2}}
\put(24,144){\line(1,-4){16}} \put(24,144){\line(3,-4){108}}
\put(24,144){\line(2,-1){162}}
\multiput(24,144)(5,0){16}{\circle*{2}} \put(60,0){\line(2,1){52}}
\put(60,0){\line(3,4){108}} \put(60,0){\line(1,5){36}}
\multiput(60,0)(-1.25,5){16}{\circle*{2}}
\put(96,183){\line(-3,-4){40.9}} \put(96,183){\line(1,-5){36}}
\multiput(96,183)(3,-4){17}{\circle*{2}}
\put(132,0){\line(1,4){15.7}}
\multiput(132,0)(-4,2){16}{\circle*{2}}
\put(168,144){\line(-1,0){64.5}}
\multiput(168,144)(-1.25,-5){16}{\circle*{2}}
\put(186,63){\line(-3,4){39.4}}
\multiput(186,63)(-4,-2){19}{\circle*{2}}

\end{picture}

\begin{picture}(200,180)(-80,0)
\put(-80,130){The hypercube sporadic maximal cyclotomic signed
graph $S_{16}$:}

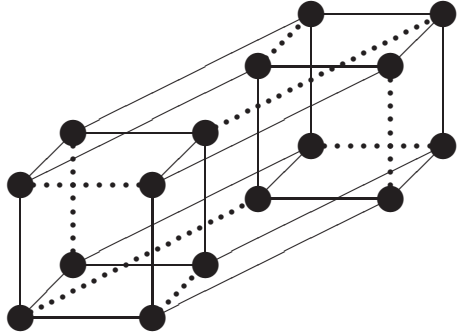
\multiput(0,0)(50,0){2}{\circle*{10}}
\multiput(0,50)(50,0){2}{\circle*{10}}
\multiput(20,20)(50,0){2}{\circle*{10}}
\multiput(20,70)(50,0){2}{\circle*{10}}
\multiput(90,45)(50,0){2}{\circle*{10}}
\multiput(90,95)(50,0){2}{\circle*{10}}
\multiput(110,65)(50,0){2}{\circle*{10}}
\multiput(110,115)(50,0){2}{\circle*{10}}

\put(0,0){\line(1,0){50}} \multiput(0,0)(4,2){23}{\circle*{2}}
\put(0,0){\line(1,1){20}} \put(0,0){\line(0,1){50}}
\multiput(0,50)(5,0){10}{\circle*{2}} \put(0,50){\line(2,1){90}}
\put(0,50){\line(1,1){20}} \put(20,20){\line(1,0){50}}
\put(20,20){\line(2,1){90}} \multiput(20,20)(0,5){10}{\circle*{2}}
\put(20,70){\line(1,0){50}} \put(20,70){\line(2,1){90}}
\put(50,0){\line(2,1){90}} \multiput(50,0)(3,3){6}{\circle*{2}}
\put(50,0){\line(0,1){50}} \put(50,50){\line(2,1){90}}
\put(50,50){\line(1,1){20}} \put(70,20){\line(2,1){90}}
\put(70,20){\line(0,1){50}} \multiput(70,70)(4,2){23}{\circle*{2}}
\put(90,45){\line(1,0){50}} \put(90,45){\line(1,1){20}}
\put(90,45){\line(0,1){50}} \put(90,95){\line(1,0){50}}
\multiput(90,95)(3,3){6}{\circle*{2}}
\multiput(110,65)(5,0){10}{\circle*{2}}
\put(110,65){\line(0,1){50}} \put(110,115){\line(1,0){50}}
\put(140,45){\line(1,1){20}}
\multiput(140,45)(0,5){10}{\circle*{2}}
\put(140,95){\line(1,1){20}} \put(160,65){\line(0,1){50}}

\end{picture}

\begin{picture}(300,150)
\put(0,130){The family $T_{2k}$ of $2k$-vertex maximal cyclotomic
toral tesselations,} \put(-25,117){for $k\ge3$ (the two copies of
vertices $A$ and $B$ should be identified):}

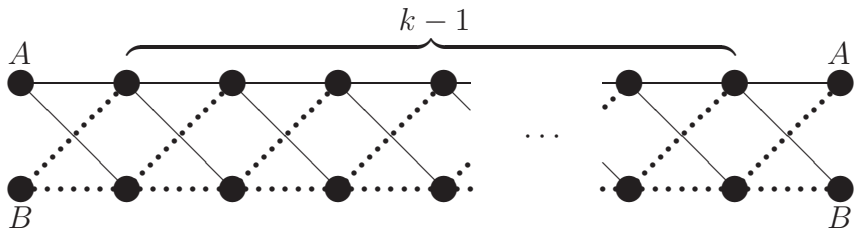
\put(0,87){$A$} \put(0,25){$B$} \put(310,87){$A$}
\put(310,25){$B$}

\multiput(5,40)(40,0){5}{\circle*{10}}
\multiput(5,80)(40,0){5}{\circle*{10}}
\multiput(235,40)(40,0){3}{\circle*{10}}
\multiput(235,80)(40,0){3}{\circle*{10}}

\multiput(5,40)(5,0){35}{\circle*{2}}
\multiput(225,40)(5,0){18}{\circle*{2}}
\put(5,80){\line(1,0){170}} \put(225,80){\line(1,0){90}}
\multiput(5,80)(40,0){4}{\line(1,-1){40}}
\multiput(235,80)(40,0){2}{\line(1,-1){40}} \put(195,57){$\cdots$}
\multiput(5,40)(3,3){13}{\circle*{2}}
\multiput(45,40)(3,3){13}{\circle*{2}}
\multiput(85,40)(3,3){13}{\circle*{2}}
\multiput(125,40)(3,3){13}{\circle*{2}}
\multiput(235,40)(3,3){13}{\circle*{2}}
\multiput(275,40)(3,3){13}{\circle*{2}}

\put(45,90){$\overbrace{\hspace{230pt}}$} \put(148,100){$k-1$}

\put(165,80){\line(1,-1){10}} \put(235,40){\line(-1,1){10}}
\multiput(165,40)(3,3){4}{\circle*{2}}
\multiput(235,80)(-3,-3){4}{\circle*{2}}
\end{picture}

\begin{picture}(300,140)

\put(0,120){$T_6$ is a signed octahedron:}

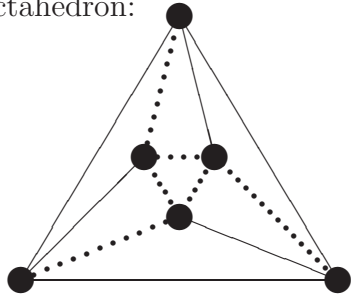
\multiput(90,20)(120,0){2}{\circle*{10}}
\multiput(150,44)(0,76){2}{\circle*{10}}
\multiput(136.7,66.7)(26.6,0){2}{\circle*{10}}

\put(90,20){\line(1,0){120}}
\multiput(90,20)(5,2){12}{\circle*{2}}
\put(90,20){\line(1,1){46.7}} \put(90,20){\line(3,5){60}}
\multiput(136.7,66.7)(2,-3.4){8}{\circle*{2}}
\multiput(136.7,66.7)(5,0){6}{\circle*{2}}
\multiput(136.7,66.7)(1.25,5){12}{\circle*{2}}
\put(150,44){\line(5,-2){60}}
\multiput(150,44)(2,3.4){8}{\circle*{2}}
\put(150,120){\line(1,-4){13.3}} \put(150,120){\line(3,-5){60}}
\multiput(163.3,66.7)(3,-3){15}{\circle*{2}}
\end{picture}

\begin{picture}(300,190)
\put(0,170){The families of $2k$-vertex cyclotomic cylindrical
tesselations $C_{2k}^{++}$} \put(-25,157){and $C_{2k}^{+-}$, for
$k\ge 2$:}

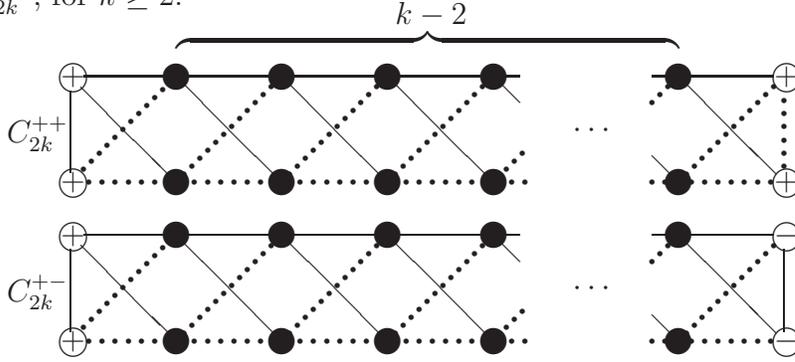
\put(75,140){$\overbrace{\hspace{190pt}}$} \put(158,150){$k-2$}

\multiput(30,26.5)(0,40){2}{\textcircled{+}}
\multiput(30,86.5)(0,40){2}{\textcircled{+}}
\multiput(75,30)(40,0){4}{\circle*{10}}
\multiput(75,70)(40,0){4}{\circle*{10}}
\multiput(75,90)(40,0){4}{\circle*{10}}
\multiput(75,130)(40,0){4}{\circle*{10}}
\multiput(265,30)(0,40){2}{\circle*{10}}
\multiput(265,90)(0,40){2}{\circle*{10}}
\multiput(300,26.5)(0,40){2}{\textcircled{$-$}}
\multiput(300,86.5)(0,40){2}{\textcircled{+}}

\multiput(42,30)(5,0){34}{\circle*{2}}
\put(40,70){\line(1,0){165}}
\multiput(42,90)(5,0){34}{\circle*{2}}
\put(40,130){\line(1,0){165}}
\multiput(255,30)(5,0){10}{\circle*{2}}
\put(255,70){\line(1,0){46}}
\multiput(255,90)(5,0){10}{\circle*{2}}
\put(255,130){\line(1,0){46}} \put(225,47){$\cdots$}
\put(225,107){$\cdots$}

\multiput(35,33.5)(0,60){2}{\line(0,1){31}}
\put(305,33.5){\line(0,1){31}}
\multiput(305,96)(0,5){6}{\circle*{2}}

\put(11,45){$C_{2k}^{+-}$} \put(11,105){$C_{2k}^{++}$}

\multiput(39,66)(0,60){2}{\line(1,-1){37}}
\multiput(75,70)(40,0){3}{\line(1,-1){40}}
\multiput(75,130)(40,0){3}{\line(1,-1){40}}
\multiput(195,70)(0,60){2}{\line(1,-1){10}}
\multiput(255,40)(0,60){2}{\line(1,-1){10}}
\multiput(265,70)(0,60){2}{\line(1,-1){37}}

\multiput(40,35)(3,3){13}{\circle*{2}}
\multiput(40,95)(3,3){13}{\circle*{2}}
\multiput(75,30)(3,3){13}{\circle*{2}}
\multiput(75,90)(3,3){13}{\circle*{2}}
\multiput(115,30)(3,3){13}{\circle*{2}}
\multiput(115,90)(3,3){13}{\circle*{2}}
\multiput(155,30)(3,3){13}{\circle*{2}}
\multiput(155,90)(3,3){13}{\circle*{2}}
\multiput(197,32)(3,3){4}{\circle*{2}}
\multiput(197,92)(3,3){4}{\circle*{2}}
\multiput(255,60)(3,3){3}{\circle*{2}}
\multiput(255,120)(3,3){3}{\circle*{2}}
\multiput(265,30)(3,3){13}{\circle*{2}}
\multiput(265,90)(3,3){13}{\circle*{2}}
\end{picture}

\begin{picture}(300,100)
\put(0,95){$C_4^{++}$ and $C_4^{+-}$ are tetrahedra:}

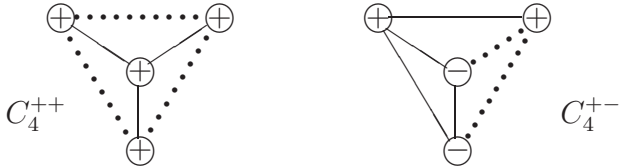
\put(40,40){$C_4^{++}$} \put(250,40){$C_4^{+-}$}

\multiput(85,26.5)(0,30){2}{\textcircled{+}}
\multiput(55,76.5)(60,0){2}{\textcircled{+}}
\multiput(205,26.5)(0,30){2}{\textcircled{$-$}}
\multiput(175,76.5)(60,0){2}{\textcircled{+}}

\multiput(64.5,71.7)(3,-5){8}{\circle*{2}}
\put(65,77){\line(3,-2){22}}
\multiput(67,80)(5,0){10}{\circle*{2}} \put(90,34){\line(0,1){20}}
\multiput(95,37.3)(3,5){8}{\circle*{2}}
\put(94,62.3){\line(3,2){22}}

\put(183,75){\line(3,-5){25}} \put(184,76){\line(3,-2){23}}
\put(185,80){\line(1,0){50}} \put(210,34){\line(0,1){20}}
\multiput(215,35)(2.7,4.5){9}{\circle*{2}}
\multiput(217,63)(4.5,3){5}{\circle*{2}}

\end{picture}

\section{Remarks concerning $T_{2k}$, $C_{2k}^{++}$, $C_{2k}^{+-}$ and their subgraphs}

\subsection{Definitions}

We say that a signed graph $G$ {\it has a profile} if its vertex set can be
partitioned into a sequence of $k\ge 3$ subsets $V_1,\dots,V_k$ so that either
\bei\item two vertices are adjacent if and only if for some $i$ one
belongs to $V_i$ and the other to $V_{i+1}$

or

\item two vertices are adjacent if and only if for some $i$ one
belongs to $V_i$ and the other to $V_{i+1}$ or one belongs to
$V_k$ and the other to $V_1$. \eni

In the latter case we say that the profile is {\it cycling}. We call the subsets $V_i$ the {\it columns} of the profile. For our
application we are interested only in profiles where the $V_i$
contain one or two vertices. In particular, this applies to
 $T_{2k}$, which has a cycling profile, and to its connected subgraphs. For a vertex $v$ in a $2$-vertex column, the other vertex in that column will be denoted $\overline{v}$, the {\it conjugate} of $v$.

It will be convenient to extend these definitions to charged signed
graphs, insisting further that each $V_i$ contains only neutral
vertices or only charged vertices all with the same charge, and
relaxing the adjacency rule to read that $xy$ is an edge in $G$ if
and only if {\it either} $x$ and $y$ are in adjacent columns {\it
or} are charged vertices in the same column. With this definition,
both $C_{2k}^{++}$ and $C_{2k}^{+-}$ have a profile.

If $G$ has a profile, then we define its {\it rank} to be the number
of columns in the profile -- this may depend on the profile, but see
Lemma \ref{L:cols} below.

In any graph $G$ we define, following \cite{HR}, a \emph{chordless
path} or {\it chordless cycle} to be a path or cycle $P$ with the property that if two vertices of $P$ are adjacent in $G$  then they are adjacent in $P$.     We call the
maximum number of vertices, taken over all chordless paths and chordless cycles of $G$, to be the \emph{(path) rank} of
$G$.     As we shall show, the two definitions of rank coincide,
except for small graphs.

\subsection{The subgraphs of $T_{2k}$, $C_{2k}^{++}$ and
$C_{2k}^{+-}$}

  As drawn above, the $2k$
vertices of $T_{2k}$ or $C_{2k}^{++}$ or $C_{2k}^{+-}$ have a
profile of rank $k$, with each column comprising two mutually
conjugate vertices.
 Now let $G$ be a connected subgraph of one of $T_{2k}$, $C_{2k}^{++}$,
or $C_{2k}^{+-}$, drawn as above. Then $G$ inherits a profile from
one of these graphs.

\begin{lemma}\label{L:cols}
Let $G$ be  equivalent to  a connected subgraph     of one of
$T_{2k}$, $C_{2k}^{++}$, and $C_{2k}^{+-}$.   If $G$ has  path
rank at least $5$ then  this equals its profile rank, and its
columns are uniquely determined. Moreover their order is
determined up to reversal or cycling.
\end{lemma}

We remark that the lemma is false if `5' is replaced by `4': the
graph $T_6$ has path rank $4$, but profile rank $3$ ; the graph
$T_8$ has
 path/profile  rank $4$, but  the columns of its profile are not
uniquely determined: if $\{1,5\}$, $\{2,6\}$, $\{3,7\}$, $\{4,8\}$
is a profile for $T_8$, then so is $\{1,7\}$, $\{2,4\}$,
$\{3,5\}$, $\{6,8\}$.  Note also that the profile does not
always determine $G$ up to equivalence, even for high rank: for
signed cycles of fixed even length there are two equivalence
classes, indistinguishable by their profiles.

\begin{proof}   We start with a chordless path or cycle $P$ of maximal
number of vertices  $r$. Because $r\ge 5$, no two of these
vertices can be in the same column (this is  not necessarily true
for $r\le 4$), and each column of $G$ contains exactly one vertex
in $P$. This shows that the profile rank equals the path rank. The
columns of the profile of $P$ inherited from that of $G$ are
singletons. We then add vertices to these columns, to complete the
profile of $G$. Because $r\ge 5$ the column to which a new vertex
can be added is completely determined by the vertices it is
adjacent to in $G$. The last sentence of the Lemma is clear.
\end{proof}

\begin{proposition}\label{P:T-C-equivalences}
\bei \item[(i)]  Let $H$ be a signed graph of rank at least $5$ that has, for some $k$, the same underlying graph as a subgraph of $T_{2k}$, drawn as above. Then $H$ is equivalent to a subgraph $G$ of $T_{2k}$ if and only if

 \bei
\item[\textbullet] the  hourglass $4$-cycles (with underlying
graph
\begin{picture}(25,15)(0,10) \multiput(0,0)(20,0){2}{\circle*{5}}
\multiput(0,20)(20,0){2}{\circle*{5}}
\multiput(0,0)(0,20){2}{\line(1,0){20}} \put(0,0){\line(1,1){20}}
\put(0,20){\line(1,-1){20}}
\end{picture} ) all have an even number of positive edges;
\item[\textbullet] the  parallelogram $4$-cycles (
\begin{picture}(45,15)(-3,10) \multiput(20,0)(20,0){2}{\circle*{5}} \multiput(0,20)(20,0){2}{\circle*{5}}
\multiput(20,0)(-20,20){2}{\line(1,0){20}}
\multiput(0,20)(20,0){2}{\line(1,-1){20}} \end{picture} ,
\begin{picture}(45,15)(-3,10) \multiput(0,0)(20,0){2}{\circle*{5}}
\multiput(20,20)(20,0){2}{\circle*{5}}
\multiput(0,0)(20,20){2}{\line(1,0){20}}
\multiput(0,0)(20,0){2}{\line(1,1){20}}
\end{picture} ) all have an odd
number of positive edges;
\item[\textbullet] the triangular $4$-cycles (
\begin{picture}(45,15)(-3,10) \multiput(0,20)(20,0){3}{\circle*{5}}
\put(20,0){\circle*{5}} \put(0,20){\line(1,0){40}}
\put(0,20){\line(1,-1){20}} \put(20,0){\line(1,1){20}}
\end{picture} ,
\begin{picture}(45,15)(0,10) \multiput(0,0)(20,0){3}{\circle*{5}} \put(20,20){\circle*{5}}
\put(0,0){\line(1,0){40}} \put(0,0){\line(1,1){20}}
\put(20,20){\line(1,-1){20}} \end{picture} ) all have an odd
number of positive edges. \eni \item[(ii)]  Let $H$ be a
charged signed graph of rank at least $5$ that has, for some
$k$, the same underlying graph as a subgraph of $C_{2k}^{++}$ or
$C_{2k}^{+-}$, drawn as above. Then $H$ is equivalent to a
subgraph $G$ of $C_{2k}^{++}$ or $C_{2k}^{+-}$ if and only if

 \bei
\item[\textbullet] the  hourglass $4$-cycles all have an even number
of positive edges;
\item[\textbullet] the  parallelogram $4$-cycles
all have an odd number of positive edges;
\item[\textbullet] the  triangular $4$-cycles
all have an odd number of positive edges;

\item[\textbullet] the triangles containing two charged vertices
in the subgraph have the property that if the charges are positive
(respectively negative)  then the triangle has an even number of
positive (resp.~negative) edges. \eni\eni
\end{proposition}

Note that we are not claiming that the charged signed subgraph $G$
found has the same underlying drawn graph as the charged signed
graph $H$ that we started with.

 \begin{proof}  We first show that the conditions given in
the Proposition are  necessary. Since $H$ has rank at least $5$, its profile is, by Lemma \ref{L:cols}, uniquely determined, and thus each $4$-cycle is specified by our standard drawing of $T_{2k}$ as being either
\bei\item an hourglass

or
\item a parallelogram $4$-cycle or a triangular $4$-cycle.
\eni

(If conjugate vertices are interchanged in the drawing, parallelogram $4$-cycles can become triangular $4$-cycles, and vice versa.) Since equivalence preserves the parity of the number of positive edges on an even cycle, we have necessity.

We now prove sufficiency. We assume that the given conditions
hold, and prove that they are sufficient: that our given graph is
then equivalent to a subgraph of $T_{2k}$, $C_{2k}^{++}$ or
$C_{2k}^{+-}$. To do this, we need to embed a graph equivalent to
$H$ into one of $T_{2k}$, $C_{2k}^{++}$ or $C_{2k}^{+-}$  so that
the resultant embedding $G$ inherits its edge and vertex signs
from the graph it is embedded into.

\bei\item[(i)]   Suppose that we have a signed graph $H$ of rank at least $5$ that shares
the same underlying graph as a subgraph of $T_{2k}$, drawn as above,
and satisfies the three $4$-cycle conditions above.

Take a maximum-length chordless path or cycle $P$ in $H$, of rank $k'$ say.
 By switching we can arrange that $P$ either has
 \bei\item[(a)] has all edges positive;
 \item[(b)] exactly one negative edge.
 \eni
  In case (b) we can further
assume that there is in fact no choice of $P$ that has, after
switching, all edges positive. Put $k=k'$ if $P$ is a cycle, and
$k=k'+1$ if not. In case (a) we embed $P$ in the top row of
$T_{2k}$. In case (b), $P$ is a cycle, and we can assume that
$P=v_1v_2\dots v_k$, with only edge $v_kv_1$ negative. We then
embed $P$ into $T_{2k}$ so that  $v_1,\dots, v_{k-1}$ are on the
top row, and the final vertex $v_k$ is one column further along,
on the bottom row. Then the signs of $P$ and $T_{2k}$ are
consistent, the edge of negative slope having positive sign, and
vice versa. Note too that in this latter case the vertex $\bar
v_k$ cannot be present in $H$. For if it were, then the edges
joining it to $v_1$ and $v_{k-1}$ could not be of the same sign
(or we could by switching make the top row all positive, and go
back and choose the top row to be $P$), or of opposite signs (as
this would then contradict the stated triangular $4$-cycle
condition).

We can now embed into $T_{2k}$ those conjugates of $v_1,\dots,v_k$
that are present in $H$, by placing them in their appropriate
columns on the bottom row of $T_{2k}$. Note that up to two
triangular $4$-cycles in $H$ may become become parallelogram
$4$-cycles, and vice versa,  by this embedding. This induces an
embedding $G$ of $H$ into $T_{2k}$, though without the signs of
the edges yet agreeing. To achieve this agreement, we switch at
these newly embedded vertices, if necessary, to ensure that all
edges of negative slope have positive sign. We also switch at any
(there can be no more than one) vertex in the bottom row that has
no incident edge of negative slope, if necessary, to ensure that
the incident edge of positive slope has negative sign.

We next claim that, after making these switchings, all edges of
the embedding $G$ do indeed have the same sign as the edges of
$T_{2k}$. First consider an edge of $G$ of positive slope. If not
already made to have negative sign, such an edge must be part of a
triangular $4$-cycle where the two horizontal edges and the edge
of negative slope all have positive sign. Hence, by the stated
triangular $4$-cycle condition, the edge of positive slope must
have negative sign. (Note that because both the stated
parallelogram $4$-cycle condition and the triangular $4$-cycle
condition hold for $H$, the triangular $4$-cycle condition holds
for $G$.) Finally, every horizontal edge on the second row is part
of an hourglass $4$-cycle, which implies that it must have
negative sign.

\item[(ii)] Again take a maximum-length chordless path $P$ in $H$,
say of rank $k'\ge 5$. At least one of its endvertices must be
charged, or else $H$ could be embedded in $T_{2k}$, and by
equivalence we may assume that if there are two charged vertices
on $P$ then they are not both negative. By switching we may assume
that all edges of $P$ are positive, so $P$ can be embedded
sign-consistently on the top row of one of $C_{2k}^{++}$ or
$C_{2k}^{+-}$, where $k=k'$ if $P$ has two signed vertices, and
$k=k'+1$ if not. We then proceed as in (i), which ensures that all
horizontal edges, and those of positive or negative slope, have
the same sign as that of $C_{2k}^{++}$ or $C_{2k}^{+-}$ in the
embedding. Finally, the triangle condition ensures that the
vertical edges must have their signs as in $C_{2k}^{++}$  or
$C_{2k}^{+-}$. \eni \end{proof}

A significant feature of the above result is that the stated
parity conditions on the $4$-cycles are all in some sense {\it
local} conditions. The requirement that the rank should be at
least $5$ is best-possible: one can permute the vertices of $T_8$
so as to produce an isomorphic underlying graph but with the
$4$-cycle parity conditions failing to hold.

\section{Proof of Theorem \ref{T:main} }
We now prove our main result, Theorem \ref{T:main}.  We first
quickly dispose of cases where at least one entry has modulus
greater than $1$ (Sections \ref{SS:modulusge3} and
\ref{SS:modulus2}).  The meat of the proof is concerned with
adjacency matrices of charged signed graphs (Sections \ref{SS:csg}
and \ref{SS:finitesearch}).

\subsection{An entry with modulus at least 3}\label{SS:modulusge3}

Let $A$ be minimal noncyclotomic, and suppose that some entry
$a=a_{ij}$ has modulus at least 3.

If this a diagonal entry ($i=j$), then $(a)$ is a noncyclotomic
submatrix, so must equal $A$. Since we can change the signs of all
entries, we can assume that $a\ge 3$, and $A$ is a member of our
first infinite family in \S\ref{SS:minnoninfinite}.

If $i\ne j$, then $A$ has a noncyclotomic submatrix
$$\left(\begin{array}{cc}b&a\\a&c\end{array}\right)\,,$$ so this must
equal $A$. Since $A$ is minimal noncyclotomic, $|b|\leq 2$ and
$|c|\leq 2$. By permuting and switching, we can suppose that $0\le
|c|\le b\le 2$, and that if $b=c=0$ then $a>0$. We see that $A$ is
a member of one of the other infinite families in
\S\ref{SS:minnoninfinite}.

\subsection{An entry with modulus 2}\label{SS:modulus2}
Now we suppose that $A$ is minimal noncyclotomic, has all entries
below 2 in modulus, with at least one entry $a_{ij}$ having
modulus equal to $2$, and indeed working up to equivalence we may
suppose that $a_{ij}=2$.

If $i=j$, then since $(2)$ is cyclotomic $A$ has at least two
rows, and contains a submatrix
$$\left(\begin{array}{cc}2&a\\a&b\end{array}\right).$$
We can choose this submatrix such that $a\ne0$, or else $A$ would
be decomposable. No such 2-by-2 matrix is cyclotomic, so this must
be the whole of $A$. Now $|b|\leq 2$ by hypothesis, hence $A$ is a
member of our second infinite family in \S\ref{SS:minnoninfinite}.

If $i\ne j$, then $A$ contains a submatrix
$$\left(\begin{array}{cc}a&2\\2&b\end{array}\right).$$
If either $a$ or $b$ is not zero, then this matrix is not
cyclotomic, so equals $A$, which is seen to be equivalent to a
member of the second or third infinite family in
\S\ref{SS:minnoninfinite}. Finally, if $a=b=0$, then
$$\left(\begin{array}{cc}a&2\\2&b\end{array}\right)=\left(\begin{array}{cc}0&2\\2&0\end{array}\right)$$
is cyclotomic, so is not the whole of $A$. In this case, $A$ has a
submatrix equivalent to
\[
\left(\begin{array}{ccc}0&2&a\\2&0&b\\a&b&c\end{array}\right),
\]
for some $a$, $b$, $c$ (all of modulus at most 2, by hypothesis).
Moreover since $A$ is indecomposable we can choose our submatrix
such that $a$ and $b$ are not both 0.

Note that $|c|<2$, or else (using $a$, $b$ not both 0) $A$ would have
a 2-by-2 noncyclotomic submatrix.

We quickly check that all cases are equivalent to one of the eight
sporadic examples listed at the start of \S\ref{SS:minnonsporadic}
that extend
$$\left(\begin{array}{cc}
0 & 2  \\
2 & 0
\end{array}
\right).$$

\subsection{Charged signed graphs: reduction to a finite
search}\label{SS:csg}

From the classification of cyclotomic charged signed graphs
recalled in Section \ref{S:maxcyc}, we know that every connected
proper subgraph of a minimal noncyclotomic charged signed graph is
either equivalent to a subgraph of one of the sporadic examples
$S_7$, $S_8$, $S'_8$, $S_{14}$, $S_{16}$, or is equivalent to a
subgraph of one of the three infinite families $T_{2k}\quad (k\ge
3)$, $C_{2k}^{++}\quad (k\ge 2)$, $C_{2k}^{+-}\quad (k\ge 2)$. We
call a minimal noncyclotomic charged signed graph with $n$
vertices \emph{supersporadic} if it has a connected subgraph with
$n-1$ vertices that is equivalent to a subgraph of one of $S_7$,
$S_8$, $S'_8$, $S_{14}$, $S_{16}$.
It is clear that the number of supersporadic minimal noncyclotomic
signed graphs is finite, and that they can all be computed (in
principle) by running through every connected subgraph of $S_7$,
$S_8$, $S_8'$, $S_{14}$, $S_{16}$ and considering all possible
ways of adding a single vertex to the subgraph. In practice, this
computation was performed rather more efficiently, as detailed in
\S\ref{SS:finitesearch}.

There remains the problem of finding any minimal noncyclotomic
charged signed graphs that are not supersporadic. Potentially,
this is an infinite search. We shall show, however, that the
number of vertices in any such graph is at most  $10$, reducing
the problem to a finite one.

\begin{proposition}\label{P:degreebound}
Let $G$ be a connected charged signed graph with $n\ge 11$
vertices, and such that every proper  connected subgraph of $G$ is
equivalent to  a subgraph of $T_{2k}$, $C_{2k}^{++}$ or
$C_{2k}^{+-}$, for some $k$. Then $G$ is also  equivalent to  a
subgraph of some $T_{2k}$, $C_{2k}^{++}$ or $C_{2k}^{+-}$.
\end{proposition}

It follows straight from this result that a minimal noncyclotomic
charged signed graph that is not supersporadic can have no more
that $10$ vertices.

\begin{proof} Take such a $G$ . Our first aim is to show that $G$ has a
profile. Take a chordless path or cycle $P$ with a maximal number
of vertices, and let  $x$ and $y$ be the endvertices of $P$ if $P$
is a path, and let them be any two adjacent vertices of $P$ if $P$
is a cycle. If the maximal number of vertices for a chordless
cycle equals that for a chordless path, then take $P$ to be a
path.

 First note that no vertex of $G$ is adjacent to $x$ but to no other
vertex on $P$, or else we could either grow $P$ to a longer
chordless path, or replace a chordless cycle $P$ by a chordless
path of equal length.  Similarly for $y$.  It follows that
$G-\{x\}$ is connected, and since it contains at least $10$
vertices it has rank at least $5$, and hence $P$ contains at least
$5$ vertices. In the case where $P$ is a path, if there were a
vertex not on $P$ adjacent to both $x$ and $y$ but to no other
vertex in $P$, then $P$ could be extended to a longer chordless
cycle: a contradiction.  In the case where $P$ is a cycle
$yxv_1v_2\cdots$, if there were a vertex $z$ not on $P$ adjacent
to both $x$ and $y$ but to no other vertex on $P$, then
$P\cup\{z\}-\{v_2\}$ would be a proper connected subgraph of $G$
containing the triangle $zyx$, hence containing two charged
vertices, yet not equivalent to a subgraph of $C_{2k}^{++}$ or
$C_{2k}^{+-}$ (here using that $P-\{v_2\}$ contains at least $4$
vertices): another contradiction. In all cases we see that $G$ has
no vertex adjacent to either $x$ or $y$ that is not also adjacent
to some vertex on $P$, so $G-\{x,y\}$ is connected.

Now $G-\{x,y\}$     has rank say $r$ at least $5$ so that, by
Lemma \ref{L:cols} , it has a uniquely determined profile.
As the profiles of $G-\{x\}$ and $G-\{y\}$ are also uniquely
determined, they can each be obtained by adding $y$ or $x$ to the
profile of $G-\{x,y\}$. Leaving aside for the moment the issue of
whether or not $x$ and $y$ are adjacent in $G$, we see that all
other possible adjacencies of $x$ in $G$ can be read off from the
profile of $G-\{y\}$, and all other possible adjacencies of $y$ in
$G$ can be read off from the profile of $G-\{x\}$. Thus we can
merge the profiles of $G-\{x\}$ and $G-\{y\}$ to obtain a new
sequence of `columns', $\mathcal{C}$, which we shall show is in
fact the profile of $G$ . In this merging, certainly columns
$2,3,\dots,r-1$ carry over unchanged. The vertices $x$ and $y$
will lie either in columns $1$ or $r$ or in a new singleton column
to the left or right of the profile of $G-\{x,y\}$. As $x$ and $y$
are the endpoints of a maximal chordless path or cycle, they must
be in opposite end columns.

First suppose that at least one of the vertices in the column of
$x$ is adjacent to at least one of the vertices in the column of
$y$. Then by deleting   column  $3$ of $G-\{x,y\}$ we obtain
another connected proper subgraph of $G$, which is therefore
equivalent to a subgraph of some $T_{2k}$, $C_{2k}^{++}$ or
$C_{2k}^{+-}$. Furthermore, it is determined by its profile, so
that in fact every vertex in the column of $x$ is adjacent to
every vertex in the column of $y$.   Thus $\mathcal{C}$ is
indeed the profile of $G$. The local conditions of Proposition
\ref{P:T-C-equivalences} are seen to hold, and $G$ is
equivalent to a subgraph of some $T_{2k}$.

Alternatively, it may be that no vertex in the column of $x$ is
adjacent to any other vertex in the column of $y$.  Then again
$\mathcal{C}$ is the profile of $G$.  The local conditions of
Proposition \ref{P:T-C-equivalences} (perhaps this time with
charged vertices at one or both ends) hold, since they hold for
both $G-\{x\}$ and $G-\{y\}$,  so that $G$ is equivalent to  a
subgraph of some $T_{2k}$, $C_{2k}^{++}$ or $C_{2k}^{+-}$.
\end{proof}

\subsection{Charged signed graphs: details of the finite search}\label{SS:finitesearch}

  By the \emph{degree} of an uncharged vertex $v$ we mean the
number of adjacent vertices (those vertices $w$ such that there is
an edge of either sign between $v$ and $w$). We define the degree
of a charged vertex to be one more than the number of adjacent
vertices.

\begin{lemma}\label{L:5} Up to equivalence, the graphs $5b$,
$5y$ and $6c$ are the only minimal noncyclotomic charged signed
graphs containing a vertex of degree greater than $4$.
\end{lemma}

\begin{proof} All such graphs $G$ with up to $6$ vertices are readily
found by searching. Now  suppose that $G$ has at least $7$
vertices, with  a vertex $x$ of degree at least $5$. Suppose first
that $x$ is uncharged. Take $x_1$, \ldots, $x_5$ to be five of the
neighbours of $x$. Since $G$ has at least $7$ vertices, the
subgraph induced by $x$, $x_1$, \ldots, $x_5$ must be cyclotomic.
Similarly if $x$ is charged, the subgraph induced by $x$ and four
of its neighbours must be cyclotomic. Yet no cyclotomic charged
signed graph has a vertex of degree greater than $4$. \end{proof}

We generated lists of connected charged signed graphs with small
numbers of vertices, and either cyclotomic or minimal
noncyclotomic, starting with the list of the two inequivalent
$1$-vertex charged signed graphs (both of these are cyclotomic).
Having produced a list of $r$-vertex charged signed graphs, to
compute a list of $r+1$-vertex charged signed graphs we considered
all ways of adding a vertex to all the $r$-vertex cyclotomics;
this would either give a cyclotomic (which was stored in the new
list), or a noncyclotomic, and in the latter case minimality was
tested, and minimal noncyclotomics added to the new list.

Since we were working up to equivalence, we deleted any graphs
from our lists that were quickly found to be equivalent to another
one. The quick test that we used did not always identify
equivalent graphs. Repeats of cyclotomic graphs were tolerated;
when producing final lists of minimal noncyclotomic graphs,
possible repeats were investigated by hand.

The search went exhaustively up to $15$-vertex charged signed
graphs, making use of Lemma \ref{L:5} once the number of vertices
was large enough. The connected minimal noncyclotomics found were
precisely those displayed in Section \ref{S:minnon} below.  After
 Proposition \ref{P:degreebound}, we then needed only  to
consider supergraphs of subgraphs of $S_{16}$, where the subgraph
contained at least $15$ vertices, but we found no more minimal
noncyclotomics.

One consequence of this search is that it reveals all those
minimal noncyclotomics that are not supersporadic: $3a$, $3f$,
$4d$, $4f$, $4n$, $5f$, $5y$, $6k$, $6l$. In particular, the
`$11$' in Proposition \ref{P:degreebound} can be reduced to
`$7$'.

\section{The minimal indecomposable noncyclotomic matrices}\label{S:minnon}

\subsection{The infinite families}\label{SS:minnoninfinite}

For any natural number $n$ greater than 2,
\[
\left( n \right)
\]
is minimal noncyclotomic. Its single eigenvalue is $n\ge2$, and
its Mahler measure is $(n+ \sqrt{n^2-4})/2\geq2.618$.
\bigskip

The following are infinite families of minimal noncyclotomic matrices:
\[
\left(
\begin{array}{cc}
2 & a \\
a & b
\end{array}
\right)\,\quad |a|\geq 1,|b|\leq 2\,,
\]
\[
\left(
\begin{array}{cc}
1 & a \\
a & b
\end{array}
\right)\,\quad |a|\geq 2,|b|\leq 1\,,
\]
\[
\left(
\begin{array}{cc}
0 & a \\
a & 0
\end{array}
\right)\,\quad a\geq 3\,.
\]
One readily checks that in all cases the spectral radius is at
least $\sqrt{2}$ and the Mahler measure is greater than $1.722$.

\subsection{The sporadic examples}\label{SS:minnonsporadic}
First we list representatives of the equivalence
classes of matrices that extend $\left(\begin{array}{cc}0&2\\2&0\end{array}\right)$:
\[
\left(\begin{array}{ccc}
0 & 2 & 2 \\
2 & 0 & 2 \\
2 & 2 & 0
\end{array}
\right), \left(\begin{array}{ccc}
0 & 2 & 2 \\
2 & 0 & 1 \\
2 & 1 & 0
\end{array}
\right), \left(\begin{array}{ccc}
0 & 2 & 1 \\
2 & 0 & 1 \\
1 & 1 & 0
\end{array}
\right), \left(\begin{array}{ccc}
0 & 2 & 0 \\
2 & 0 & 2 \\
0 & 2 & 0
\end{array}
\right),
\]
\[
\left(\begin{array}{ccc}
0 & 2 & 1 \\
2 & 0 & 1 \\
1 & 1 & 1
\end{array}
\right), \left(\begin{array}{ccc}
0 & 2 & 1 \\
2 & 0 & 1 \\
1 & 1 & -1
\end{array}
\right), \left(\begin{array}{ccc}
0 & 2 & 0 \\
2 & 0 & 1 \\
0 & 1 & 1
\end{array}
\right), \left(\begin{array}{ccc}
0 & 2 & 0 \\
2 & 0 & 1 \\
0 & 1 & 0
\end{array}
\right).
\]

In each case, the spectral radius is at least $\sqrt{5}$ and the
Mahler measure is greater than $2.081$.

The remaining sporadic examples are all adjacency matrices of
charged signed graphs. We display them in order of their number of
vertices.  Then in Table \ref{Ta:minimal} we give their spectral radius and Mahler measure.

\begin{picture}(300,60)
\put(0,42){$3$-vertex minimal noncyclotomic charged signed
graphs:}

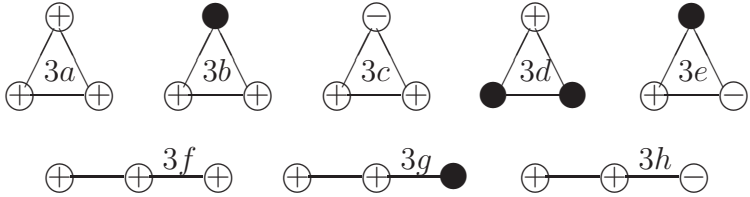
\multiput(5,-3.5)(30,0){6}{\textcircled{+}}
\multiput(190,0)(30,0){2}{\circle*{10}}
\put(245,-3.5){\textcircled{+}} \put(275,-3.5){\textcircled{$-$}}
\put(20,26.5){\textcircled{+}} \put(85,30){\circle*{10}}
\put(140,26.5){\textcircled{$-$}} \put(200,26.5){\textcircled{+}}
\put(265,30){\circle*{10}}

\put(20,5){$3a$} \put(80,5){$3b$} \put(140,5){$3c$}
\put(200,5){$3d$} \put(260,5){$3e$}

\multiput(15,0)(60,0){5}{\line(1,0){20}}
\multiput(13,3)(60,0){5}{\line(1,2){11}}
\multiput(37,3)(60,0){5}{\line(-1,2){11}}

\end{picture}

\begin{picture}(300,30)
\multiput(20,-3.5)(30,0){5}{\textcircled{+}}
\put(175,0){\circle*{10}}
\multiput(200,-3.5)(30,0){2}{\textcircled{+}}
\put(260,-3.5){\textcircled{$-$}}

\put(65,3){$3f$} \put(155,3){$3g$} \put(245,3){$3h$}

\multiput(30,0)(30,0){2}{\line(1,0){20}}
\multiput(120,0)(30,0){2}{\line(1,0){20}}
\multiput(210,0)(30,0){2}{\line(1,0){20}}
\end{picture}

\begin{picture}(300,80)
\put(0,62){$4$-vertex minimal noncyclotomic charged signed
graphs:}

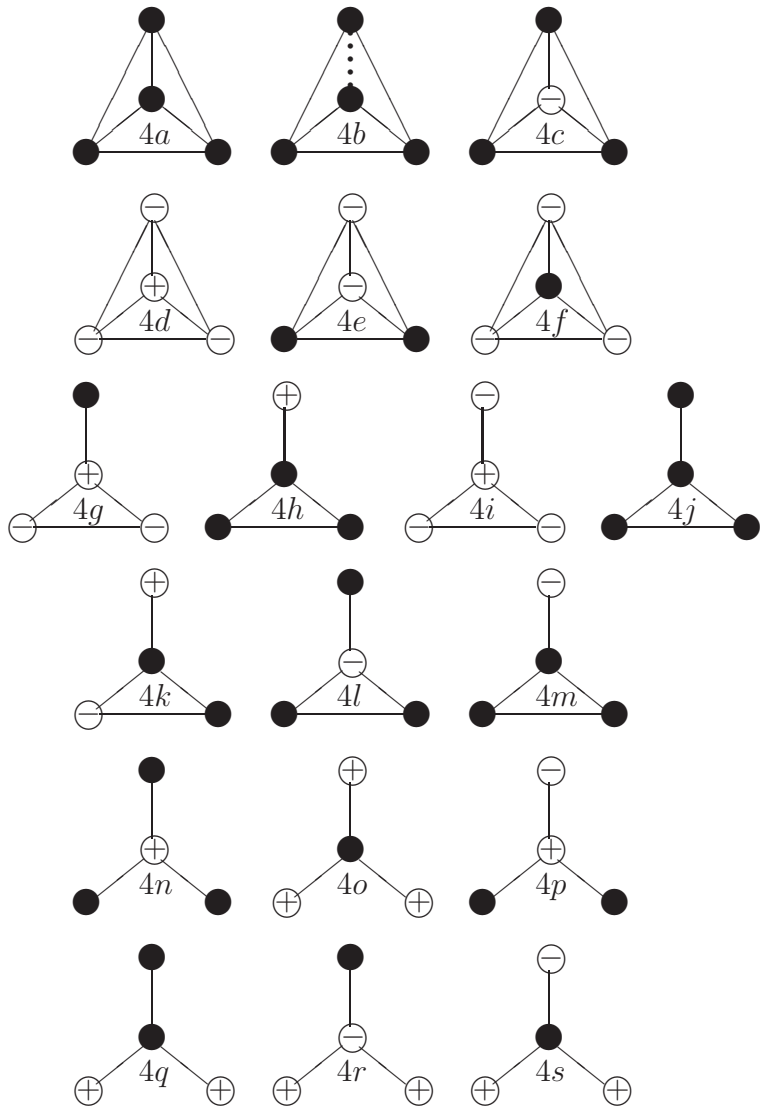
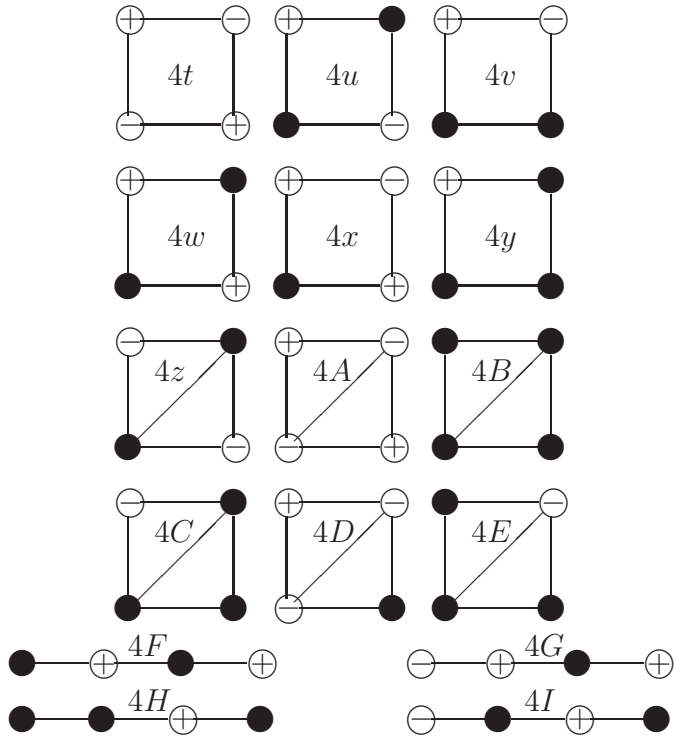
\multiput(50,0)(75,0){3}{\circle*{10}}
\multiput(100,0)(75,0){3}{\circle*{10}}
\multiput(75,20)(75,0){2}{\circle*{10}}
\put(220,16.5){\textcircled{$-$}}
\multiput(75,50)(75,0){3}{\circle*{10}}

\put(70,3){$4a$} \put(145,3){$4b$} \put(220,3){$4c$}

\multiput(50,0)(75,0){3}{\line(1,0){50}}
\multiput(50,0)(75,0){3}{\line(5,4){22}}
\multiput(100,0)(75,0){3}{\line(-5,4){21}}
\multiput(50,0)(75,0){3}{\line(1,2){25}}
\multiput(100,0)(75,0){3}{\line(-1,2){25}}
\put(75,20){\line(0,1){30}} \multiput(150,20)(0,5){6}{\circle*{2}}
\put(225,24){\line(0,1){26}}

\end{picture}

\begin{picture}(300,70)
\multiput(45,-3.5)(50,0){2}{\textcircled{$-$}}
\multiput(125,0)(50,0){2}{\circle*{10}}
\multiput(195,-3.5)(50,0){2}{\textcircled{$-$}}
\put(70,16.5){\textcircled{+}} \put(145,16.5){\textcircled{$-$}}
\put(225,20){\circle*{10}}
\multiput(70,46.5)(75,0){3}{\textcircled{$-$}}

\put(70,3){$4d$} \put(145,3){$4e$} \put(220,3){$4f$}

\multiput(55,0)(75,0){3}{\line(1,0){40}}
\multiput(54,2)(75,0){3}{\line(5,4){18}}
\multiput(53,3)(75,0){3}{\line(1,2){21}}
\multiput(75,24)(75,0){3}{\line(0,1){21}}
\multiput(96,2)(75,0){3}{\line(-5,4){17.5}}
\multiput(97,3)(75,0){3}{\line(-1,2){21}}
\end{picture}

\begin{picture}(300,70)
\multiput(20,-3.5)(50,0){2}{\textcircled{$-$}}
\multiput(100,0)(50,0){2}{\circle*{10}}
\multiput(170,-3.5)(50,0){2}{\textcircled{$-$}}
\multiput(250,0)(50,0){2}{\circle*{10}}
\multiput(45,16.5)(150,0){2}{\textcircled{+}}
\multiput(125,20)(150,0){2}{\circle*{10}}
\multiput(50,50)(225,0){2}{\circle*{10}}
\put(120,46.5){\textcircled{+}} \put(195,46.5){\textcircled{$-$}}

\put(45,3){$4g$} \put(120,3){$4h$} \put(195,3){$4i$}
\put(270,3){$4j$}

\multiput(30,0)(75,0){4}{\line(1,0){40}}
\multiput(29,2)(75,0){4}{\line(5,4){18}}
\multiput(50,24)(75,0){4}{\line(0,1){21}}
\multiput(71,2)(75,0){4}{\line(-5,4){17.5}}
\end{picture}

\begin{picture}(300,70)
\put(45,-3.5){\textcircled{$-$}}
\multiput(100,0)(75,0){3}{\circle*{10}}
\multiput(125,0)(75,0){2}{\circle*{10}}
\multiput(75,20)(150,0){2}{\circle*{10}}
\put(145,16.5){\textcircled{$-$}} \put(70,46.5){\textcircled{+}}
\put(150,50){\circle*{10}} \put(220,46.5){\textcircled{$-$}}

\put(70,3){$4k$} \put(145,3){$4l$} \put(220,3){$4m$}

\multiput(55,0)(75,0){3}{\line(1,0){45}}
\multiput(54,2)(75,0){3}{\line(5,4){18}}
\multiput(96,2)(75,0){3}{\line(-5,4){17.5}}
\multiput(75,24)(75,0){3}{\line(0,1){21}}
\end{picture}

\begin{picture}(300,70)
\multiput(50,0)(50,0){2}{\circle*{10}}
\multiput(120,-3.5)(50,0){2}{\textcircled{+}}
\multiput(200,0)(50,0){2}{\circle*{10}}
\multiput(70,16.5)(150,0){2}{\textcircled{+}}
\put(150,20){\circle*{10}} \put(75,50){\circle*{10}}
\put(145,46.5){\textcircled{+}} \put(220,46.5){\textcircled{$-$}}

\put(70,3){$4n$} \put(145,3){$4o$} \put(220,3){$4p$}

\multiput(54,2)(75,0){3}{\line(5,4){18}}
\multiput(96,2)(75,0){3}{\line(-5,4){17.5}}
\multiput(75,24)(75,0){3}{\line(0,1){21}}
\end{picture}

\begin{picture}(300,70)
\multiput(45,-3.5)(75,0){3}{\textcircled{+}}
\multiput(95,-3.5)(75,0){3}{\textcircled{+}}
\multiput(75,20)(150,0){2}{\circle*{10}}
\put(145,16.5){\textcircled{$-$}}
\multiput(75,50)(75,0){2}{\circle*{10}}
\put(220,46.5){\textcircled{$-$}}

\put(70,3){$4q$} \put(145,3){$4r$} \put(220,3){$4s$}

\multiput(54,2)(75,0){3}{\line(5,4){18}}
\multiput(96,2)(75,0){3}{\line(-5,4){17.5}}
\multiput(75,24)(75,0){3}{\line(0,1){21}}
\end{picture}

\begin{picture}(300,60)
\multiput(65,-3.5)(100,0){2}{\textcircled{$-$}}
\put(105,-3.5){\textcircled{+}} \put(130,0){\circle*{10}}
\multiput(190,0)(40,0){2}{\circle*{10}}
\multiput(65,36.5)(60,0){3}{\textcircled{+}}
\multiput(105,36.5)(120,0){2}{\textcircled{$-$}}
\put(170,40){\circle*{10}}

\put(85,15){$4t$} \put(145,15){$4u$} \put(205,15){$4v$}

\multiput(75,0)(60,0){3}{\line(1,0){30}}
\multiput(70,3.5)(60,0){3}{\line(0,1){31}}
\multiput(75,40)(60,0){3}{\line(1,0){30}}
\multiput(110,3.5)(60,0){3}{\line(0,1){31}}
\end{picture}

\begin{picture}(300,60)
\multiput(70,0)(60,0){3}{\circle*{10}}
\multiput(105,-3.5)(60,0){2}{\textcircled{+}}
\put(230,0){\circle*{10}}
\multiput(65,36.5)(60,0){3}{\textcircled{+}}
\multiput(110,40)(120,0){2}{\circle*{10}}
\put(165,36.5){\textcircled{$-$}}

\put(85,15){$4w$} \put(145,15){$4x$} \put(205,15){$4y$}

\multiput(75,0)(60,0){3}{\line(1,0){30}}
\multiput(70,3.5)(60,0){3}{\line(0,1){31}}
\multiput(75,40)(60,0){3}{\line(1,0){30}}
\multiput(110,3.5)(60,0){3}{\line(0,1){31}}
\end{picture}

\begin{picture}(300,60)
\put(70,0){\circle*{10}}
\multiput(105,-3.5)(20,0){2}{\textcircled{$-$}}
\put(165,-3.5){\textcircled{+}}
\multiput(190,0)(40,0){2}{\circle*{10}}
\multiput(65,36.5)(100,0){2}{\textcircled{$-$}}
\put(110,40){\circle*{10}} \put(125,36.5){\textcircled{+}}
\multiput(190,40)(40,0){2}{\circle*{10}}

\put(80,25){$4z$} \put(140,25){$4A$} \put(200,25){$4B$}

\multiput(75,0)(60,0){3}{\line(1,0){30}}
\multiput(70,3.5)(60,0){3}{\line(0,1){31}}
\multiput(75,40)(60,0){3}{\line(1,0){30}}
\multiput(110,3.5)(60,0){3}{\line(0,1){31}}
\multiput(73,2)(60,0){3}{\line(1,1){34}}
\end{picture}

\begin{picture}(300,60)
\multiput(70,0)(120,0){2}{\circle*{10}}
\multiput(110,0)(60,0){3}{\circle*{10}}
\put(125,-3.5){\textcircled{$-$}} \put(65,36.5){\textcircled{$-$}}
\multiput(110,40)(80,0){2}{\circle*{10}}
\put(125,36.5){\textcircled{+}}
\multiput(165,36.5)(60,0){2}{\textcircled{$-$}}

\put(80,25){$4C$} \put(140,25){$4D$} \put(200,25){$4E$}

\multiput(75,0)(60,0){3}{\line(1,0){30}}
\multiput(70,3.5)(60,0){3}{\line(0,1){31}}
\multiput(75,40)(60,0){3}{\line(1,0){30}}
\multiput(110,3.5)(60,0){3}{\line(0,1){31}}
\multiput(73,2)(60,0){3}{\line(1,1){34}}
\end{picture}

\begin{picture}(300,20)
\multiput(30,0)(60,0){2}{\circle*{10}}
\multiput(55,-3.5)(60,0){2}{\textcircled{+}}
\put(175,-3.5){\textcircled{$-$}}
\multiput(205,-3.5)(60,0){2}{\textcircled{+}}
\put(240,0){\circle*{10}}

\put(70,3){$4F$} \put(220,3){$4G$}

\multiput(35,0)(30,0){3}{\line(1,0){20}}
\multiput(185,0)(30,0){3}{\line(1,0){20}}
\end{picture}

\begin{picture}(300,20)
\multiput(30,0)(30,0){2}{\circle*{10}}
\put(85,-3.5){\textcircled{+}} \put(120,0){\circle*{10}}
\put(175,-3.5){\textcircled{$-$}}
\multiput(210,0)(60,0){2}{\circle*{10}}
\put(235,-3.5){\textcircled{+}}

\put(70,3){$4H$} \put(220,3){$4I$}

\multiput(35,0)(30,0){3}{\line(1,0){20}}
\multiput(185,0)(30,0){3}{\line(1,0){20}}
\end{picture}

\begin{picture}(300,100)
\put(0,65){$5$-vertex minimal noncyclotomic charged signed
graphs:}

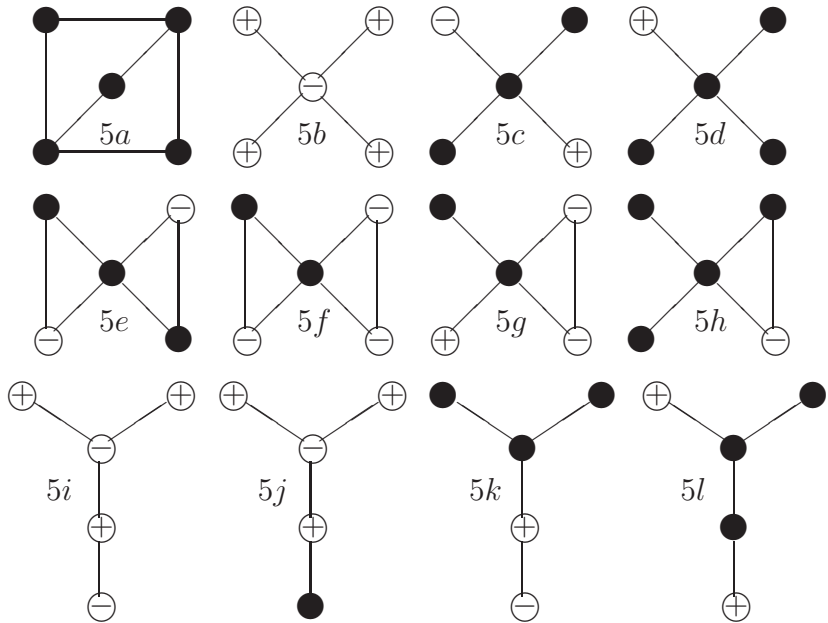
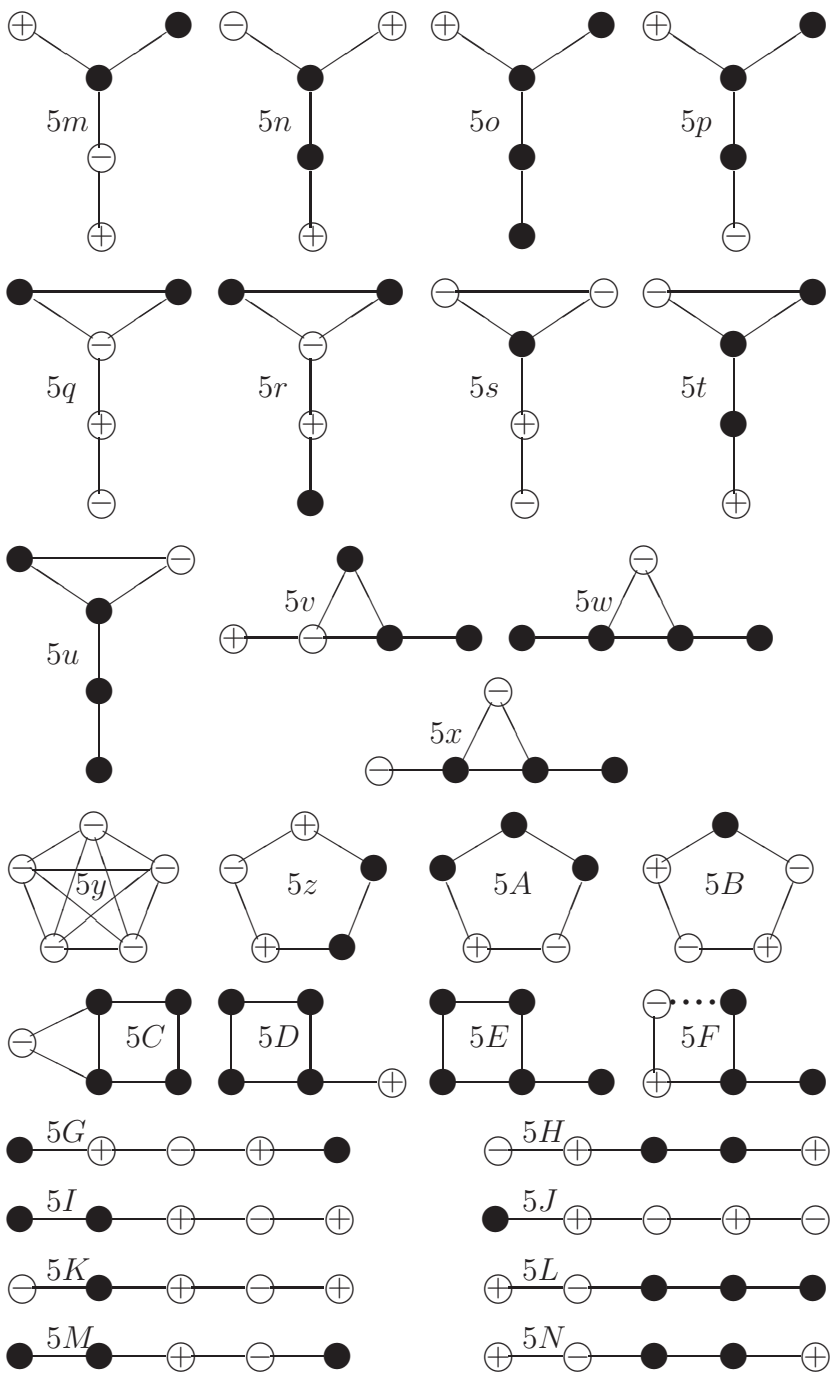
\multiput(10,0)(50,0){2}{\circle*{10}}
\multiput(80,-3.5)(50,0){2}{\textcircled{+}}
\put(160,0){\circle*{10}} \put(205,-3.5){\textcircled{+}}
\multiput(235,0)(50,0){2}{\circle*{10}} \put(35,25){\circle*{10}}
\put(105,21.5){\textcircled{$-$}}
\multiput(185,25)(75,0){2}{\circle*{10}}
\multiput(10,50)(50,0){2}{\circle*{10}}
\multiput(80,46.5)(50,0){2}{\textcircled{+}}
\put(155,46.5){\textcircled{$-$}}
\multiput(210,50)(75,0){2}{\circle*{10}}
\put(230,46.5){\textcircled{+}}

\put(30,3){$5a$} \put(105,3){$5b$} \put(180,3){$5c$}
\put(255,3){$5d$}

\multiput(10,0)(0,50){2}{\line(1,0){50}}
\multiput(10,0)(50,0){2}{\line(0,1){50}}
\multiput(13,3)(75,0){4}{\line(1,1){19}}
\multiput(38,28)(75,0){4}{\line(1,1){19}}
\multiput(89,46)(75,0){3}{\line(1,-1){19}}
\multiput(114,21)(75,0){3}{\line(1,-1){19}}
\end{picture}

\begin{picture}(300,70)
\put(5,-3.5){\textcircled{$-$}} \put(60,0){\circle*{10}}
\multiput(80,-3.5)(50,0){2}{\textcircled{$-$}}
\put(155,-3.5){\textcircled{+}}
\multiput(205,-3.5)(75,0){2}{\textcircled{$-$}}
\multiput(35,25)(75,0){4}{\circle*{10}}
\multiput(10,50)(75,0){4}{\circle*{10}}
\multiput(55,46.5)(75,0){3}{\textcircled{$-$}}
\put(285,50){\circle*{10}} \put(235,0){\circle*{10}}

\put(30,3){$5e$} \put(105,3){$5f$} \put(180,3){$5g$}
\put(255,3){$5h$}

\multiput(10,4)(50,0){2}{\line(0,1){41}}
\multiput(85,4)(50,0){2}{\line(0,1){41}}
\multiput(210,4)(75,0){2}{\line(0,1){41}}
\multiput(13,3)(75,0){4}{\line(1,1){19}}
\multiput(38,28)(75,0){4}{\line(1,1){19}}
\multiput(14,46)(75,0){4}{\line(1,-1){19}}
\multiput(39,21)(75,0){4}{\line(1,-1){19}}
\end{picture}

\begin{picture}(300,100)

\multiput(25,-3.5)(160,0){2}{\textcircled{$-$}}
\put(110,0){\circle*{10}} \put(265,-3.5){\textcircled{+}}
\multiput(25,26.5)(80,0){3}{\textcircled{+}}
\put(270,30){\circle*{10}} \put(10,40){$5i$} \put(90,40){$5j$}
\put(170,40){$5k$} \put(250,40){$5l$}
\multiput(25,56.5)(80,0){2}{\textcircled{$-$}}
\multiput(190,60)(80,0){2}{\circle*{10}}
\multiput(-5,76.5)(60,0){2}{\textcircled{+}}
\multiput(75,76.5)(60,0){2}{\textcircled{+}}
\multiput(160,80)(60,0){2}{\circle*{10}}
\put(235,76.5){\textcircled{+}} \put(300,80){\circle*{10}}

\multiput(30,3.5)(80,0){4}{\line(0,1){21}}
\multiput(30,33.5)(80,0){4}{\line(0,1){21}}
\multiput(33.5,62.5)(80,0){4}{\line(3,2){22}}
\multiput(27.5,62.5)(80,0){4}{\line(-3,2){22}}
\end{picture}

\begin{picture}(300,100)
\multiput(25,-3.5)(80,0){2}{\textcircled{+}}
\put(190,0){\circle*{10}} \put(265,-3.5){\textcircled{$-$}}
\put(25,26.5){\textcircled{$-$}}
\multiput(110,30)(80,0){3}{\circle*{10}} \put(10,40){$5m$}
\put(90,40){$5n$} \put(170,40){$5o$} \put(250,40){$5p$}
\multiput(30,60)(80,0){4}{\circle*{10}}
\multiput(-5,76.5)(140,0){2}{\textcircled{+}}
\multiput(60,80)(160,0){2}{\circle*{10}}
\put(75,76.5){\textcircled{$-$}}
\multiput(155,76.5)(80,0){2}{\textcircled{+}}
\put(300,80){\circle*{10}}

\multiput(30,3.5)(80,0){4}{\line(0,1){21}}
\multiput(30,33.5)(80,0){4}{\line(0,1){21}}
\multiput(33.5,62.5)(80,0){4}{\line(3,2){22}}
\multiput(27.5,62.5)(80,0){4}{\line(-3,2){22}}
\end{picture}

\begin{picture}(300,100)
\multiput(25,-3.5)(160,0){2}{\textcircled{$-$}}
\put(110,0){\circle*{10}} \put(265,-3.5){\textcircled{+}}
\multiput(25,26.5)(80,0){3}{\textcircled{+}}
\put(270,30){\circle*{10}} \put(10,40){$5q$} \put(90,40){$5r$}
\put(170,40){$5s$} \put(250,40){$5t$}
\multiput(25,56.5)(80,0){2}{\textcircled{$-$}}
\multiput(190,60)(80,0){2}{\circle*{10}}
\multiput(0,80)(60,0){2}{\circle*{10}}
\multiput(80,80)(60,0){2}{\circle*{10}}
\multiput(155,76.5)(60,0){2}{\textcircled{$-$}}
\put(235,76.5){\textcircled{$-$}} \put(300,80){\circle*{10}}

\multiput(5,80)(80,0){4}{\line(1,0){50}}
\multiput(30,3.5)(80,0){4}{\line(0,1){21}}
\multiput(30,33.5)(80,0){4}{\line(0,1){21}}
\multiput(33.5,62.5)(80,0){4}{\line(3,2){22}}
\multiput(27.5,62.5)(80,0){4}{\line(-3,2){22}}
\end{picture}

\begin{picture}(300,100)
\put(30,0){\circle*{10}} \put(130,-3.5){\textcircled{$-$}}
\multiput(165,0)(30,0){3}{\circle*{10}} \put(155,10){$5x$}
\put(30,30){\circle*{10}} \put(10,40){$5u$}
\put(175,26.5){\textcircled{$-$}} \put(75,46.5){\textcircled{+}}
\put(105,46.5){\textcircled{$-$}}
\multiput(140,50)(30,0){2}{\circle*{10}}
\multiput(190,50)(30,0){4}{\circle*{10}} \put(30,60){\circle*{10}}
\put(100,60){$5v$} \put(210,60){$5w$}
\multiput(0,80)(125,0){2}{\circle*{10}}
\multiput(55,76.5)(175,0){2}{\textcircled{$-$}}

\multiput(140,0)(30,0){3}{\line(1,0){20}}
\multiput(85,50)(30,0){3}{\line(1,0){20}}
\multiput(195,50)(30,0){3}{\line(1,0){20}}
\put(5,80){\line(1,0){50}} \put(30,0){\line(0,1){60}}
\put(33.5,62.5){\line(3,2){22}} \put(26.5,62.5){\line(-3,2){22}}
\put(167.5,3.5){\line(1,2){11}}
\multiput(112.5,53.5)(110,0){2}{\line(1,2){11}}
\put(193,3.5){\line(-1,2){11}}
\multiput(138,53.5)(110,0){2}{\line(-1,2){11}}
\end{picture}

\begin{picture}(300,66)
\multiput(7,-3.5)(30,0){2}{\textcircled{$-$}}
\put(87,-3.5){\textcircled{+}} \put(122,0){\circle*{10}}
\multiput(167,-3.5)(110,0){2}{\textcircled{+}}
\multiput(197,-3.5)(50,0){2}{\textcircled{$-$}}
\multiput(-5,26.5)(54,0){2}{\textcircled{$-$}}
\multiput(75,26.5)(214,0){2}{\textcircled{$-$}}
\put(134,30){\circle*{10}}
\multiput(160,30)(54,0){2}{\circle*{10}}
\put(235,26.5){\textcircled{+}} \put(22,42.5){\textcircled{$-$}}
\put(102,42.5){\textcircled{+}}
\multiput(187,46)(80,0){2}{\circle*{10}}

\put(13,4){\line(1,3){12.5}} \put(41,4){\line(-1,3){12.5}}
\put(15,2){\line(5,4){35}} \put(39,2){\line(-5,4){35}}

\put(21,20){$5y$} \put(101,20){$5z$} \put(179,20){$5A$}
\put(259,20){$5B$} \multiput(10,3)(80,0){4}{\line(-2,5){8.5}}
\multiput(44,3)(80,0){4}{\line(2,5){8.5}}
\multiput(3,32)(80,0){4}{\line(5,3){19.5}}
\multiput(51,32)(80,0){4}{\line(-5,3){19.5}}
\end{picture}

\begin{picture}(300,50)
\multiput(30,0)(30,0){2}{\circle*{10}}
\multiput(80,0)(30,0){2}{\circle*{10}}
\put(135,-3.5){\textcircled{+}}
\multiput(160,0)(30,0){3}{\circle*{10}}
\put(235,-3.5){\textcircled{+}}
\multiput(270,0)(30,0){2}{\circle*{10}}
\put(-5,11.5){\textcircled{$-$}}
\multiput(30,30)(30,0){2}{\circle*{10}}
\multiput(80,30)(30,0){2}{\circle*{10}}
\multiput(160,30)(30,0){2}{\circle*{10}}
\put(235,26.5){\textcircled{$-$}} \put(270,30){\circle*{10}}

\put(30,0){\line(1,0){30}} \put(80,0){\line(1,0){55}}
\put(160,0){\line(1,0){60}} \put(245,0){\line(1,0){55}}
\multiput(30,30)(50,0){2}{\line(1,0){30}}
\put(160,30){\line(1,0){30}}
\multiput(247,30)(5,0){5}{\circle*{2}} \put(5,80){\line(1,0){44}}
\multiput(17,50)(80,0){4}{\line(1,0){20}}
\multiput(30,0)(30,0){2}{\line(0,1){30}}
\multiput(80,0)(30,0){2}{\line(0,1){30}}
\multiput(160,0)(30,0){2}{\line(0,1){30}}
\put(240,3.5){\line(0,1){21}} \put(270,0){\line(0,1){30}}
\put(4,12){\line(2,-1){25}} \put(4,17){\line(2,1){25}}

\put(40,13){$5C$} \put(90,13){$5D$} \put(170,13){$5E$}
\put(250,13){$5F$}
\end{picture}

\begin{picture}(300,25)
\multiput(0,0)(120,0){2}{\circle*{10}}
\multiput(25,-3.5)(60,0){2}{\textcircled{+}}
\put(55,-3.5){\textcircled{$-$}} \put(175,-3.5){\textcircled{$-$}}
\multiput(205,-3.5)(90,0){2}{\textcircled{+}}
\multiput(240,0)(30,0){2}{\circle*{10}}

\put(10,3){$5G$} \put(190,3){$5H$}

\multiput(5,0)(30,0){4}{\line(1,0){20}}
\multiput(185,0)(30,0){4}{\line(1,0){20}}
\end{picture}

\begin{picture}(300,25)
\multiput(0,0)(30,0){2}{\circle*{10}}
\multiput(55,-3.5)(60,0){2}{\textcircled{+}}
\put(85,-3.5){\textcircled{$-$}} \put(180,0){\circle*{10}}
\multiput(205,-3.5)(60,0){2}{\textcircled{+}}
\multiput(235,-3.5)(60,0){2}{\textcircled{$-$}}

\put(10,3){$5I$} \put(190,3){$5J$}

\multiput(5,0)(30,0){4}{\line(1,0){20}}
\multiput(185,0)(30,0){4}{\line(1,0){20}}
\end{picture}

\begin{picture}(300,25)
\multiput(-5,-3.5)(90,0){2}{\textcircled{$-$}}
\put(30,0){\circle*{10}}
\multiput(55,-3.5)(60,0){2}{\textcircled{+}}
\put(175,-3.5){\textcircled{+}} \put(205,-3.5){\textcircled{$-$}}
\multiput(240,0)(30,0){3}{\circle*{10}}

\put(10,3){$5K$} \put(190,3){$5L$}

\multiput(5,0)(30,0){4}{\line(1,0){20}}
\multiput(185,0)(30,0){4}{\line(1,0){20}}
\end{picture}

\begin{picture}(300,25)
\multiput(0,0)(30,0){2}{\circle*{10}}
\put(55,-3.5){\textcircled{+}} \put(85,-3.5){\textcircled{$-$}}
\put(120,0){\circle*{10}}
\multiput(175,-3.5)(120,0){2}{\textcircled{+}}
\put(205,-3.5){\textcircled{$-$}}
\multiput(240,0)(30,0){2}{\circle*{10}}

\put(10,3){$5M$} \put(190,3){$5N$}

\multiput(5,0)(30,0){4}{\line(1,0){20}}
\multiput(185,0)(30,0){4}{\line(1,0){20}}
\end{picture}

\begin{picture}(320,110)
\put(0,75){$6$-vertex minimal noncyclotomic charged signed
graphs:}

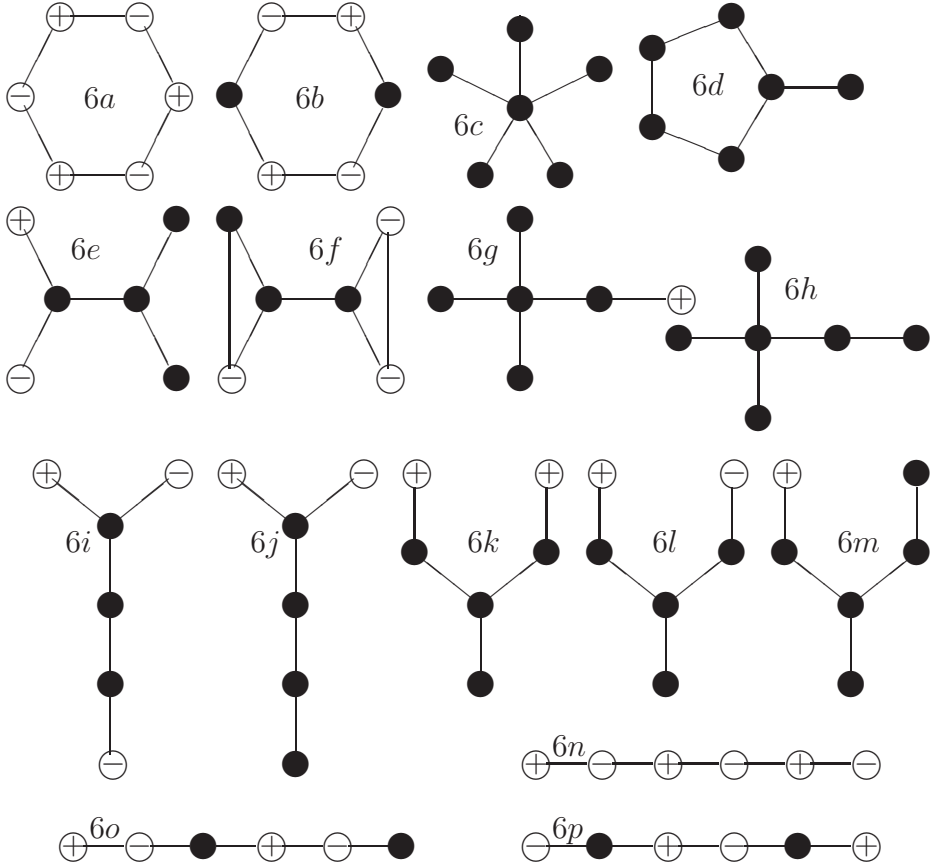
\multiput(-10,-3.5)(0,60){2}{\textcircled{+}}
\multiput(20,-3.5)(0,60){2}{\textcircled{$-$}}
\put(-25,26.5){\textcircled{$-$}} \put(35,26.5){\textcircled{+}}
\put(5,25){$6a$}

\multiput(70,-3.5)(30,60){2}{\textcircled{+}}
\multiput(100,-3.5)(-30,60){2}{\textcircled{$-$}}
\multiput(60,30)(60,0){2}{\circle*{10}} \put(85,25){$6b$}

\multiput(155,0)(30,0){2}{\circle*{10}}
\multiput(170,25)(0,30){2}{\circle*{10}}
\multiput(140,40)(60,0){2}{\circle*{10}} \put(145,15){$6c$}

\multiput(220,18)(0,30){2}{\circle*{10}}
\multiput(250,6)(0,54){2}{\circle*{10}}
\multiput(265,33)(30,0){2}{\circle*{10}} \put(235,30){$6d$}

\multiput(0,0)(80,0){2}{\line(1,0){20}}
\multiput(0,60)(80,0){2}{\line(1,0){20}}
\put(265,33){\line(1,0){30}}

\put(170,25){\line(0,1){35}} \put(220,18){\line(0,1){30}}
\put(140,40){\line(2,-1){30}} \put(170,25){\line(2,1){30}}
\put(155,0){\line(3,5){15}} \put(185,0){\line(-3,5){15}}

\put(220,48){\line(5,2){30}} \put(220,18){\line(5,-2){30}}
\put(250,60){\line(3,-5){15}} \put(250,6){\line(3,5){15}}

\multiput(-17.5,26)(80,0){2}{\line(1,-2){10.5}}
\multiput(28,56)(80,0){2}{\line(1,-2){10.5}}
\multiput(27,3)(80,0){2}{\line(1,2){11}}
\multiput(-18,33)(80,0){2}{\line(1,2){11}}
\end{picture}

\begin{picture}(300,91)

\put(-25,11.5){\textcircled{$-$}} \put(40,15){\circle*{10}}
\multiput(-5,45)(30,0){2}{\circle*{10}}
\put(-25,71.5){\textcircled{+}} \put(40,75){\circle*{10}}
\put(0,60){$6e$}

\multiput(55,11.5)(60,0){2}{\textcircled{$-$}}
\multiput(75,45)(30,0){2}{\circle*{10}} \put(60,75){\circle*{10}}
\put(115,71.5){\textcircled{$-$}} \put(90,60){$6f$}

\put(170,15){\circle*{10}}
\multiput(140,45)(30,0){3}{\circle*{10}}
\put(225,41.5){\textcircled{+}} \put(170,75){\circle*{10}}
\put(150,60){$6g$}

\put(260,0){\circle*{10}} \multiput(230,30)(30,0){4}{\circle*{10}}
\put(260,60){\circle*{10}} \put(270,45){$6h$}

\multiput(-18,19)(80,0){2}{\line(1,2){11}}
\multiput(27,48)(80,0){2}{\line(1,2){11}}

\multiput(-17.5,71)(80,0){2}{\line(1,-2){10.5}}
\multiput(26.5,41)(80,0){2}{\line(1,-2){10.5}}

\multiput(-5,45)(80,0){2}{\line(1,0){30}}
\put(140,45){\line(1,0){85}} \put(230,30){\line(1,0){90}}
\multiput(60,18.5)(60,0){2}{\line(0,1){51.5}}
\put(170,15){\line(0,1){60}} \put(260,0){\line(0,1){60}}
\end{picture}

\begin{picture}(320,130)

\multiput(170,-3.5)(50,0){3}{\textcircled{+}}
\multiput(195,-3.5)(50,0){3}{\textcircled{$-$}} \put(182,3){$6n$}
\put(10,-3.5){\textcircled{$-$}} \put(85,0){\circle*{10}}
\multiput(15,30)(70,0){5}{\circle*{10}}
\multiput(15,60)(70,0){5}{\circle*{10}}
\multiput(130,80)(70,0){3}{\circle*{10}}
\multiput(180,80)(70,0){3}{\circle*{10}}
\multiput(15,90)(70,0){2}{\circle*{10}}
\multiput(-15,106.5)(70,0){5}{\textcircled{+}}
\multiput(35,106.5)(70,0){2}{\textcircled{$-$}}
\put(175,106.5){\textcircled{+}}
\put(245,106.5){\textcircled{$-$}} \put(320,110){\circle*{10}}

\put(-2,80){$6i$} \put(68,80){$6j$} \put(150,80){$6k$}
\put(220,80){$6l$} \put(290,80){$6m$}

\multiput(180,0)(25,0){5}{\line(1,0){15}}
\put(15,3.5){\line(0,1){86.5}} \put(85,0){\line(0,1){90}}
\multiput(130,80)(70,0){3}{\line(0,1){24.5}}
\multiput(155,30)(70,0){3}{\line(0,1){30}}
\multiput(180,80)(70,0){3}{\line(0,1){24.5}}
\multiput(15,90)(70,0){2}{\line(5,4){22}}
\multiput(155,60)(70,0){3}{\line(5,4){25}}
\multiput(15,90)(70,0){2}{\line(-5,4){22}}
\multiput(155,60)(70,0){3}{\line(-5,4){25}}
\end{picture}

\begin{picture}(300,30)

\multiput(-5,-3.5)(75,0){2}{\textcircled{+}}
\multiput(20,-3.5)(75,0){2}{\textcircled{$-$}}
\multiput(50,0)(75,0){2}{\circle*{10}} \put(7,3){$6o$}

\multiput(170,-3.5)(75,0){2}{\textcircled{$-$}}
\multiput(200,0)(75,0){2}{\circle*{10}}
\multiput(220,-3.5)(75,0){2}{\textcircled{+}} \put(182,3){$6p$}

\multiput(5,0)(25,0){5}{\line(1,0){15}}
\multiput(180,0)(25,0){5}{\line(1,0){15}}
\end{picture}

\begin{picture}(300,130)
\put(0,102){$7$-vertex minimal noncyclotomic charged signed
graphs:}

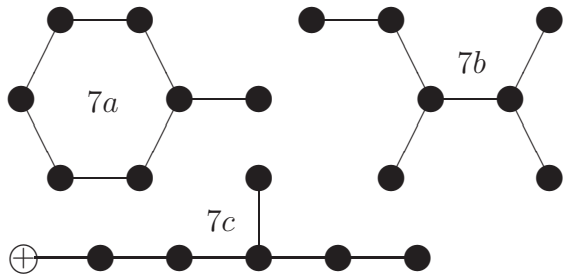
\put(45,-3.5){\textcircled{+}}
\multiput(80,0)(30,0){5}{\circle*{10}} \put(140,30){\circle*{10}}
\put(120,10){$7c$}

\multiput(65,30)(30,0){2}{\circle*{10}}
\multiput(50,60)(60,0){2}{\circle*{10}} \put(140,60){\circle*{10}}
\multiput(65,90)(30,0){2}{\circle*{10}} \put(75,55){$7a$}

\multiput(190,30)(60,0){2}{\circle*{10}}
\multiput(205,60)(30,0){2}{\circle*{10}}
\multiput(160,90)(30,0){2}{\circle*{10}}
\put(250,90){\circle*{10}} \put(215,70){$7b$}

\put(55,0){\line(1,0){145}}
\multiput(110,60)(95,0){2}{\line(1,0){30}}
\multiput(65,30)(0,60){2}{\line(1,0){30}}
\put(160,90){\line(1,0){30}} \put(140,0){\line(0,1){30}}
\multiput(50,60)(45,-30){2}{\line(1,2){15}}
\multiput(190,30)(45,30){2}{\line(1,2){15}}
\multiput(65,30)(45,30){2}{\line(-1,2){15}}
\multiput(205,60)(45,-30){2}{\line(-1,2){15}}
\end{picture}

\begin{picture}(300,100)
\put(0,75){$8$-vertex minimal noncyclotomic charged signed
graphs:}

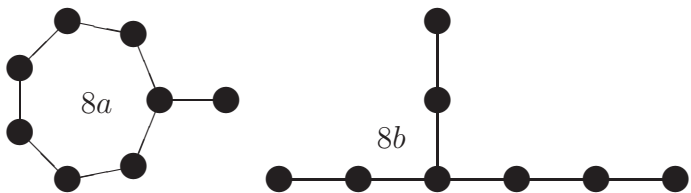
\multiput(120,0)(30,0){6}{\circle*{10}}
\multiput(180,30)(0,30){2}{\circle*{10}} \put(157,12){$8b$}

\multiput(40,0)(25,5){2}{\circle*{10}}
\multiput(22,18)(0,24){2}{\circle*{10}}
\multiput(75,30)(25,0){2}{\circle*{10}}
\multiput(40,60)(25,-5){2}{\circle*{10}} \put(45,25){$8a$}

\put(120,0){\line(1,0){150}} \put(75,30){\line(1,0){25}}
\put(22,18){\line(0,1){24}} \put(180,0){\line(0,1){60}}
\put(22,42){\line(1,1){15}} \put(40,60){\line(5,-1){25}}
\put(65,55){\line(2,-5){10}} \put(75,30){\line(-2,-5){10}}
\put(65,5){\line(-5,-1){25}} \put(40,0){\line(-1,1){15}}
\end{picture}

\begin{picture}(300,100)
\put(25,-3.5){\textcircled{+}}
\multiput(60,0)(30,0){6}{\circle*{10}} \put(150,30){\circle*{10}}
\put(127,12){$8d$}

\multiput(30,50)(30,0){6}{\circle*{10}}
\multiput(90,80)(60,0){2}{\circle*{10}} \put(67,62){$8c$}

\put(35,0){\line(1,0){175}} \put(30,50){\line(1,0){150}}
\put(150,0){\line(0,1){30}}
\multiput(90,50)(60,0){2}{\line(0,1){30}}
\end{picture}

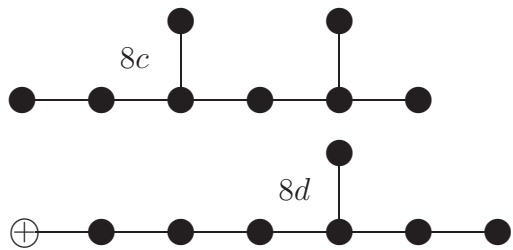
\begin{picture}(300,100)
\put(0,75){$9$-vertex minimal noncyclotomic charged signed
graphs:}

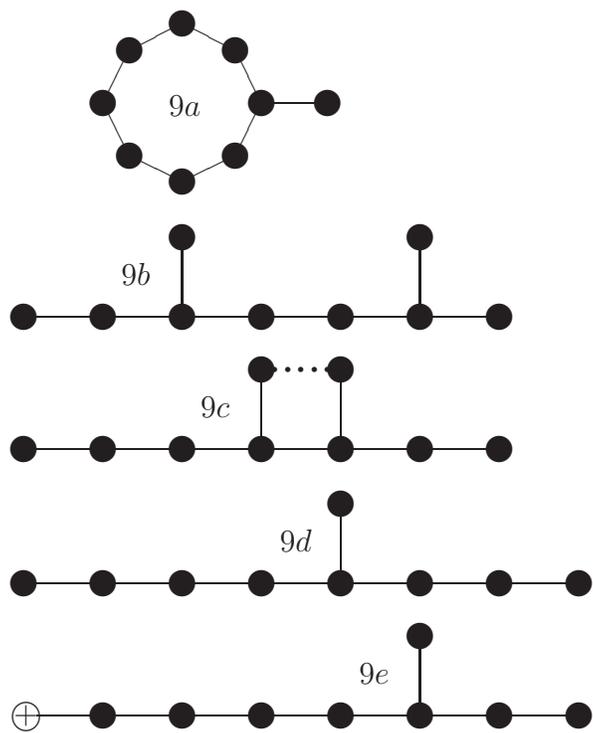
\multiput(70,30)(60,0){2}{\circle*{10}}
\multiput(80,10)(0,40){2}{\circle*{10}}
\multiput(100,0)(0,60){2}{\circle*{10}}
\multiput(120,10)(0,40){2}{\circle*{10}}
\put(155,30){\circle*{10}} \put(95,25){$9a$}

\put(130,30){\line(1,0){25}}

\multiput(70,30)(50,-20){2}{\line(1,2){10}}
\multiput(80,50)(20,-50){2}{\line(2,1){20}}
\multiput(100,60)(-20,-50){2}{\line(2,-1){20}}
\multiput(120,50)(-50,-20){2}{\line(1,-2){10}}
\end{picture}

\begin{picture}(300,100)
\multiput(40,0)(30,0){7}{\circle*{10}}
\multiput(130,30)(30,0){2}{\circle*{10}} \put(107,12){$9c$}

\multiput(40,50)(30,0){7}{\circle*{10}}
\multiput(100,80)(90,0){2}{\circle*{10}} \put(77,62){$9b$}

\multiput(40,0)(0,50){2}{\line(1,0){180}}
\multiput(130,30)(5,0){6}{\circle*{2}}
\multiput(130,0)(30,0){2}{\line(0,1){30}}
\multiput(100,50)(90,0){2}{\line(0,1){30}}
\end{picture}

\begin{picture}(300,100)
\put(35,-3.5){\textcircled{+}}
\multiput(70,0)(30,0){7}{\circle*{10}} \put(190,30){\circle*{10}}
\put(167,12){$9e$}

\multiput(40,50)(30,0){8}{\circle*{10}} \put(160,80){\circle*{10}}
\put(137,62){$9d$}

\put(45,0){\line(1,0){205}} \put(40,50){\line(1,0){210}}
\multiput(190,0)(-30,50){2}{\line(0,1){30}}
\end{picture}

\begin{picture}(300,120)
\put(0,95){$10$-vertex minimal noncyclotomic charged signed
graphs:}

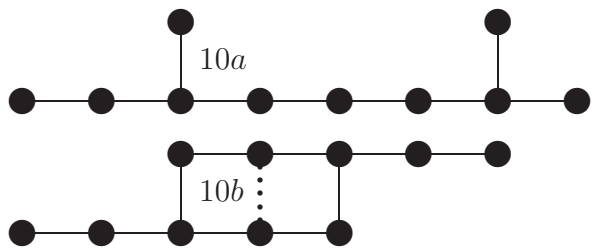
\multiput(30,0)(30,0){5}{\circle*{10}}
\multiput(90,30)(30,0){5}{\circle*{10}}
\put(30,0){\line(1,0){120}} \put(90,30){\line(1,0){120}}
\multiput(90,0)(60,0){2}{\line(0,1){30}}
\multiput(120,0)(0,5){6}{\circle*{2}} \put(97,12){$10b$}

\multiput(30,50)(30,0){8}{\circle*{10}}
\multiput(90,80)(120,0){2}{\circle*{10}}
\put(30,50){\line(1,0){210}}
\multiput(90,50)(120,0){2}{\line(0,1){30}} \put(97,62){$10a$}
\end{picture}

\begin{picture}(300,100)
\multiput(30,0)(30,0){7}{\circle*{10}}
\multiput(30,30)(30,0){3}{\circle*{10}}
\put(30,0){\line(1,0){180}} \put(30,30){\line(1,0){30}}
\multiput(60,30)(5,0){6}{\circle*{2}}
\multiput(60,0)(30,0){2}{\line(0,1){30}} \put(97,12){$10d$}

\put(25,46.5){\textcircled{+}}
\multiput(60,50)(30,0){8}{\circle*{10}} \put(210,80){\circle*{10}}
\put(35,50){\line(1,0){235}} \put(210,50){\line(0,1){30}}
\put(97,62){$10c$}
\end{picture}

\begin{picture}(300,100)
\multiput(30,0)(30,0){9}{\circle*{10}} \put(90,30){\circle*{10}}
\put(97,12){$10f$} \put(30,0){\line(1,0){240}}
\put(90,0){\line(0,1){30}}

\multiput(30,50)(30,0){8}{\circle*{10}}
\multiput(60,80)(30,0){2}{\circle*{10}}
\put(30,50){\line(1,0){210}} \multiput(60,80)(5,0){6}{\circle*{2}}
\multiput(60,50)(30,0){2}{\line(0,1){30}} \put(97,62){$10e$}
\end{picture}

\setcounter{LTchunksize}{155}
\begin{center}
\begin{longtable}{r|c|c|c|c}
 \caption[]{Spectral radius and Mahler measure of the minimal noncyclotomic charged signed graphs.
 The `outside' column gives the number of eigenvalues outside the interval $[-2,2]$.}\\
\# & Name  & Spectral radius & Outside & Mahler measure  \\
 \hline
  \endfirsthead
  \caption[]{(continued) Spectral radius and Mahler measure of the minimal noncyclotomic
   charged signed graphs.}\\
  \# & Name  & Spectral radius & Outside & Mahler measure  \\
  \hline
  \endhead
  \endfoot 
  \endlastfoot

                 1 & $3a$ & 3.00000 & 1 &  2.61803\\
                  2 & $3b$ & 2.73205 & 1 & 2.29663\\
                  3 & $3c$ & 2.56155 & 1 & 2.08102\\
                  4 & $3d$ & 2.41421 & 1 & 1.88320\\
                  5 & $3e$ & 2.21432 & 1 & 1.58235\\
                  6 & $3f$ & 2.41421 & 1 & 1.88320\\
                  7 & $3g$ & 2.24698 & 1 & 1.63557\\
                  8 & $3h$ & 2.17009 & 1 & 1.50614\\
                  9 & $4a$ & 3.00000 & 1 & 2.61803\\
                 10 & $4b$ & 2.23607 & 2 & 2.61803\\
                 11 & $4c$ & 2.79129 & 1 & 2.36921\\
                 12 & $4d$ & 2.73205 & 1 & 2.29663\\
                 13 & $4e$ & 2.56155 & 1 & 2.08102\\
                 14 & $4f$ & 2.30278 & 1 & 1.72208\\
                 15 & $4g$ & 2.30278 & 1 & 1.72208\\
                 16 & $4h$ & 2.30278 & 1 & 1.72208\\
                 17 & $4i$ & 2.21432 & 1 & 1.58235\\
                 18 & $4j$ & 2.17009 & 1 & 1.50614\\
                 19 & $4k$ & 2.14386 & 1 & 1.45799\\
                 20 & $4l$ & 2.11491 & 1 & 1.40127\\
                 21 & $4m$ & 2.11491 & 1 & 1.40127\\
                 22 & $4n$ & 2.30278 & 1 & 1.72208\\
                 23 & $4o$ & 2.30278 & 1 & 1.72208\\
                 24 & $4p$ & 2.21432 & 1 & 1.58235\\
                 25 & $4q$ & 2.17009 & 1 & 1.50614\\
                 26 & $4r$ & 2.11491 & 1 & 1.40127\\
                 27 & $4s$ & 2.11491 & 1 & 1.40127\\
                 28 & $4t$ & 2.23607 & 2 & 2.61803\\
                 29 & $4u$ & 2.23607 & 2 & 2.61803\\
                 30 & $4v$ & 2.18890 & 2 & 2.36921\\
                 31 & $4w$ & 2.56155 & 1 & 2.08102\\
                 32 & $4x$ & 2.41421 & 1 & 1.88320\\
                 33 & $4y$ & 2.34292 & 1 & 1.78164\\
                 34 & $4z$ & 2.23607 & 2 & 2.61803\\
                 35 & $4A$ & 2.56155 & 1 & 2.08102\\
                 36 & $4B$ & 2.56155 & 1 & 2.08102\\
                 37 & $4C$ & 2.41421 & 1 & 1.88320\\
                 38 & $4D$ & 2.34292 & 1 & 1.78164\\
                 39 & $4E$ & 2.30278 & 1 & 1.72208\\
                 40 & $4F$ & 2.19353 & 1 &  1.54720\\
                 41 & $4G$ & 2.12676 & 1 & 1.42501\\
                 42 & $4H$ & 2.09529 & 1 & 1.36000\\
                 43 & $4I$ & 2.06150 & 1 & 1.28064\\
                 44 & $5a$ & 2.44949 & 2 & 3.73205\\
                 45 & $5b$ & 2.23607 & 2 & 2.61803\\
                 46 & $5c$ & 2.13578 & 2 & 2.08102\\
                 47 & $5d$ & 2.19869 & 1 & 1.55603\\
                 48 & $5e$ & 2.30278 & 1 & 1.72208\\
                 49 & $5f$ & 2.17009 & 1 & 1.50614\\
                 50 & $5g$ & 2.19869 & 1 & 1.55603\\
                 51 & $5h$ & 2.17009 & 1 & 1.50614\\
                 52 & $5i$ & 2.10100 & 2 & 1.88320\\
                 53 & $5j$ & 2.15976 & 2 & 1.84752\\
                 54 & $5k$ & 2.14386 & 1 & 1.45799\\
                 55 & $5l$ & 2.13883 & 1 & 1.44842\\
                 56 & $5m$ & 2.09529 & 1 & 1.36000\\
                 57 & $5n$ & 2.09118 & 1 & 1.35098\\
                 58 & $5o$ & 2.06150 & 1 & 1.28064\\
                 59 & $5p$ & 2.03850 & 1 & 1.21639\\
                 60 & $5q$ & 2.10100 & 2 & 1.88320\\
                 61 & $5r$ & 2.15976 & 2 & 1.84752\\
                 62 & $5s$ & 2.14386 & 1 & 1.45799\\
                 63 & $5t$ & 2.11491 & 1 & 1.40127\\
                 64 & $5u$ & 2.02642 & 1 & 1.17628\\
                 65 & $5v$ & 2.13797 & 2 & 1.83505\\
                 66 & $5w$ & 2.11491 & 1 & 1.40127\\
                 67 & $5x$ & 2.06150 & 1 & 1.28064\\
                 68 & $5y$ & 3.00000 & 1 & 2.61803\\
                 69 & $5z$ & 2.34292 & 1 & 1.78164\\
                 70 & $5A$ & 2.17009 & 1 & 1.50614\\
                 71 & $5B$ & 2.22833 & 1 & 1.60545\\
                 72 & $5C$ & 2.34292 & 1 & 1.78164\\
                 73 & $5D$ & 2.25619 & 2 & 2.22371\\
                 74 & $5E$ & 2.13578 & 2 & 2.08102\\
                 75 & $5F$ & 2.05411 & 1 & 1.26123\\
                 76 & $5G$ & 2.11491 & 1 & 1.40127\\
                 77 & $5H$ & 2.10637 & 1 & 1.38364\\
                 78 & $5I$ & 2.08508 & 1 & 1.33731\\
                 79 & $5J$ & 2.06659 & 1 & 1.29349\\
                 80 & $5K$ & 2.05411 & 1 & 1.26123\\
                 81 & $5L$ & 2.04314 & 1 & 1.23039\\
                 82 & $5M$ & 2.03850 & 1 & 1.21639\\
                 83 & $5N$ & 2.02642 & 1 & 1.17628\\
                 84 & $6a$ & 2.23607 & 2 & 2.61803\\
                 85 & $6b$ & 2.21432 & 2 & 2.50382\\
                 86 & $6c$ & 2.23607 & 2 & 2.61803\\
                 87 & $6d$ & 2.11491 & 1 & 1.40127\\
                 88 & $6e$ & 2.10100 & 2 & 1.88320\\
                 89 & $6f$ & 2.12676 & 1 & 1.42501\\
                 90 & $6g$ & 2.15976 & 2 & 1.84752\\
                 91 & $6h$ & 2.07431 & 2 & 1.72208\\
                 92 & $6i$ & 2.07852 & 2 & 1.50646\\
                 93 & $6j$ & 2.02852 & 2 & 1.40127\\
                 94 & $6k$ & 2.11491 & 1 & 1.40127\\
                 95 & $6l$ & 2.02852 & 2 & 1.40127\\
                 96 & $6m$ & 2.04671 & 1 & 1.24073\\
                 97 & $6n$ & 2.06082 & 2 & 1.63557\\
                 98 & $6o$ & 2.07103 & 2 & 1.57837\\
                 99 & $6p$ & 2.04907 & 2 & 1.55603\\
                 100 & $7a$ & 2.10100 & 2 & 1.88320\\
                 101 & $7b$ & 2.05288 & 2 & 1.58235\\
                 102 & $7c$ & 2.03850 & 1 & 1.21639\\
                 103 & $8a$ & 2.09118 & 1 & 1.35098\\
                 104 & $8b$ & 2.02852 & 2 & 1.40127\\
                 105 & $8c$ & 2.04208 & 2 & 1.50614\\
                 106 & $8d$ & 2.03334 & 1 & 1.20003\\
                 107 & $9a$ & 2.08397 & 2 & 1.78164\\
                 108 & $9b$ & 2.03565 & 2 & 1.45799\\
                 109 & $9c$ & 2.02368 & 2 & 1.36000\\
                 110 & $9d$ & 2.01532 & 2 & 1.28064\\
                 111 & $9e$ & 2.02986 & 1 & 1.18837\\
                112 & $10a$ & 2.03144 & 2 & 1.42501\\
                113 & $10b$ & 2.02642 & 2 & 1.38364\\
                114 & $10c$ & 2.02739 & 2 & 1.25364\\
                115 & $10d$ & 2.01348 & 2 & 1.26123\\
                116 & $10e$ & 2.00960 & 2 & 1.21639\\
                117 & $10f$ & 2.00659 & 2 & 1.17628\\
\end{longtable}\label{Ta:minimal}
\end{center}

\section{Charged signed graphs of small spectral radius}\label{S:SSR}

Any noncyclotomic charged signed graph of spectral radius less
than $2.019$ must, by interlacing, contain as a subgraph a minimal
noncyclotomic charged signed graph of spectral radius less than
$2.019$. Thus the former can be `grown' from the latter by
successively adding a vertex of charge $-1$, $0$ or $1$, and
adjoining it in all possible ways to the vertices of the current
graph. Furthermore, we claim that we can assume that the vertex
adjoined is of degree at most $4$.

For suppose that the added vertex,  $v$ say,   is of degree at
least $5$ and that the resulting graph $G$ is of spectral radius
less than $2.019$. Consider the two cases

\begin{itemize}
\item[(i)] $v$ charged.

 Consider the  subgraph $G_5$ of
$G$ on $v$ and four of its neighbours. As no maximal  cyclotomic
graph contains a charged vertex of degree $5$, $G_5$ cannot be
cyclotomic. It therefore contains a minimal noncyclotomic subgraph
with at most $5$ vertices and spectral radius less than
$2.019$. However, from our results, there is no minimal
noncyclotomic graph of spectral radius less than $2.019$ with
fewer than $9$ vertices.

\item[(ii)] $v$ neutral.

 Consider the subgraph $G_6$ of $G$ on $v$ and five of its
neighbours. As no maximal  cyclotomic graph contains a  vertex of
degree $5$, $G_6$ cannot be cyclotomic. It therefore contains a
minimal noncyclotomic subgraph with at most $6$ vertices, and
we obtain the same contradiction.
\end{itemize}
The results of this growing procedure are shown in the pictures
below.  Together with the minimal examples of small spectral
radius, $9d$, $10d$, $10e$, $10f$, we produce Table \ref{Ta-1},
and establish Theorem \ref{T:spectral radius}.  It turns out that
all these charged signed graphs are in fact simply signed graphs.
Several are starlike trees $T_{a,b,c}$, as in Figure $1$.

\begin{picture}(300,135)
\put(0,105){Connected $10$-vertex nonminimal noncyclotomic
(charged) signed graphs} \put(-22,92){with spectral radius
$<2.019$:}

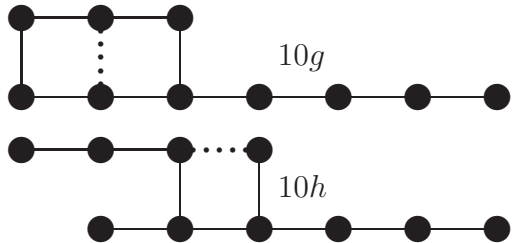
\multiput(90,0)(30,0){6}{\circle*{10}}
\multiput(60,30)(30,0){4}{\circle*{10}} \put(157,12){$10h$}
\put(90,0){\line(1,0){150}} \put(60,30){\line(1,0){60}}
\multiput(120,30)(5,0){6}{\circle*{2}}
\multiput(120,0)(30,0){2}{\line(0,1){30}}

\multiput(60,50)(30,0){7}{\circle*{10}}
\multiput(60,80)(30,0){3}{\circle*{10}} \put(157,62){$10g$}
\put(60,50){\line(1,0){180}} \put(60,80){\line(1,0){60}}
\multiput(60,50)(60,0){2}{\line(0,1){30}}
\multiput(90,50)(0,5){6}{\circle*{2}}

\end{picture}

\begin{picture}(300,100)
\put(0,85){Connected $11$-vertex nonminimal noncyclotomic
(charged) signed graphs} \put(-22,72){with spectral radius
$<2.019$:}

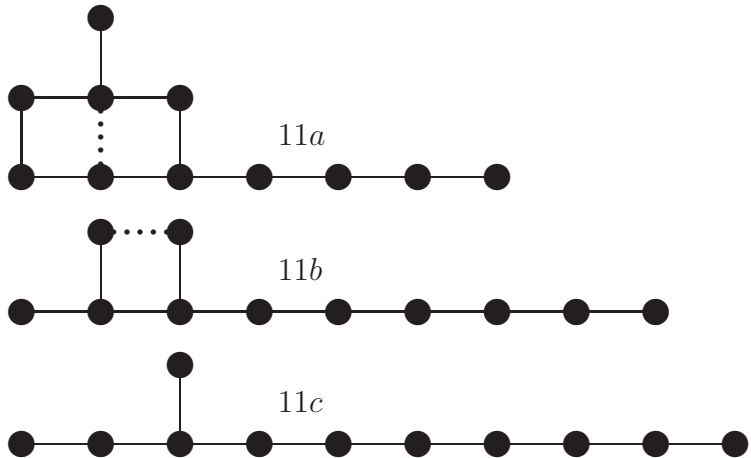
\multiput(60,0)(30,0){7}{\circle*{10}}
\multiput(60,30)(30,0){3}{\circle*{10}} \put(90,60){\circle*{10}}
\put(157,12){$11a$} \put(60,0){\line(1,0){180}}
\put(60,30){\line(1,0){60}}
\multiput(60,0)(60,0){2}{\line(0,1){30}}
\put(90,30){\line(0,1){30}} \multiput(90,0)(0,5){6}{\circle*{2}}

\end{picture}

\begin{picture}(300,100)
\multiput(60,0)(30,0){10}{\circle*{10}} \put(120,30){\circle*{10}}
\put(157,12){$11c$} \put(60,0){\line(1,0){270}}
\put(120,0){\line(0,1){30}}

\multiput(60,50)(30,0){9}{\circle*{10}}
\multiput(90,80)(30,0){2}{\circle*{10}} \put(157,62){$11b$}
\put(60,50){\line(1,0){240}} \multiput(90,80)(5,0){6}{\circle*{2}}
\multiput(90,50)(30,0){2}{\line(0,1){30}}
\end{picture}

\begin{picture}(300,130)
\put(0,105){Connected $12$-vertex nonminimal noncyclotomic
(charged) signed graphs} \put(-22,92){with spectral radius
$<2.019$:}

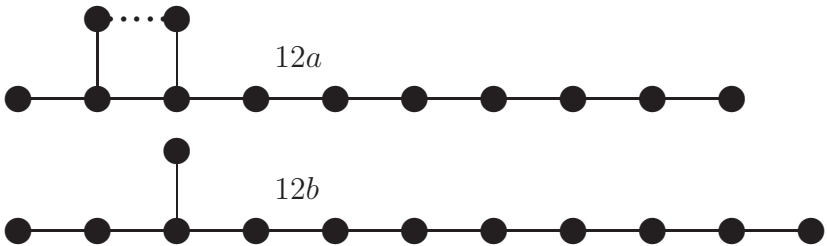
\multiput(30,50)(30,0){10}{\circle*{10}}
\multiput(60,80)(30,0){2}{\circle*{10}} \put(127,62){$12a$}
\put(30,50){\line(1,0){270}} \multiput(60,80)(5,0){6}{\circle*{2}}
\multiput(60,50)(30,0){2}{\line(0,1){30}}

\multiput(30,0)(30,0){11}{\circle*{10}} \put(90,30){\circle*{10}}
\put(127,12){$12b$} \put(30,0){\line(1,0){300}}
\put(90,0){\line(0,1){30}}

\end{picture}

\begin{picture}(300,85)
\put(0,60){Connected $13$- to $18$-vertex nonminimal noncyclotomic
(charged) signed} \put(-22,47){graphs with spectral radius
$<2.019$:}

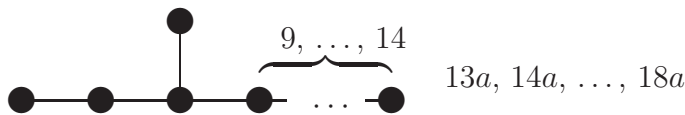
\multiput(60,0)(30,0){4}{\circle*{10}} \put(170,-5){$\cdots$}
\put(200,0){\circle*{10}} \put(120,30){\circle*{10}}
\put(60,0){\line(1,0){100}} \put(190,0){\line(1,0){10}}
\put(120,0){\line(0,1){30}}
\put(150,10){$\overbrace{\hspace{50pt}}$} \put(158,20){$9$, \dots,
$14$} \put(220,5){$13a$, $14a$, \dots, $18a$}
\end{picture}

\section{Charged signed graphs of small Mahler measure}\label{S:SMM}

Any noncyclotomic charged signed graph of Mahler measure less than
$1.3$ must, by interlacing, contain as a subgraph a minimal
noncyclotomic charged signed graph of Mahler measure less than
$1.3$. Thus the former can be grown from the latter. Again we
claim that we can assume that the vertex adjoined is of degree at
most $4$.

For suppose that $v$  is of degree at least $5$ and that the
resulting graph $G$ is of Mahler measure less than $1.3$. Again,
consider the two cases

\begin{itemize}
\item[(i)] $v$ charged.

Consider the  subgraph $G_5$ of $G$ on $v$ and four of its
neighbours. As no maximal cyclotomic graph contains a charged
vertex of degree $5$, $G_5$ cannot be cyclotomic. It therefore
contains a minimal noncyclotomic subgraph. However, from our
results, $5b$ and $5y$ are the only minimal noncyclotomic graphs
containing a vertex of degree $5$, and their Mahler measures are
all greater than $1.3$. Hence $G_5$ is not minimal noncyclotomic.
It therefore contains a minimal noncyclotomic subgraph with at
most four vertices, and Mahler measure less than $1.3$. From our
results, the only one is $4I$, which, having no vertices of degree
$4$, must be  the subgraph $G_5-\{v\}$. We now check by computer
that when $v$ is adjoined to the four vertices of $4I$ with all
$16$ possible choices of edge signs then in each case the
resulting graph has Mahler measure greater than $1.3$.

\item[(ii)] $v$ neutral.

Consider the subgraph $G_6$ of $G$ on $v$ and five of its
neighbours. Again, $G_6$ cannot be cyclotomic, and so contains a
minimal noncyclotomic subgraph. Now $G_6$ itself is not minimal,
as no $6$-vertex minimal graph of Mahler measure less than $1.3$
contains a vertex of degree greater than $3$. Suppose that $G_6$
has a minimal noncyclotomic subgraph containing $v$. It can then
have at most four vertices, so must be $4I$. But $4I$ has no
neutral vertex of degree $3$. Hence no minimal noncyclotomic
subgraph of $G_6$ contains $v$. We now check by computer that when
$v$ is adjoined to $4I$ as above (and also to another vertex if
necessary, which might itself be adjacent to vertices of $4I$), or
to the five vertices of the ten minimal $5$-vertex noncyclotomic
subgraphs of Mahler measure less than $1.3$ with all $2^5$
possible choices of edge signs, then in each case the resulting
graph has Mahler measure greater than $1.3$.

\end{itemize}

The new charged signed graphs found by this growing procedure are
shown below.  Together with the minimal examples of small Mahler
measure, $4I$, $5o$, $5p$, $5u$, $5x$, $5F$, $5J$, $5K$, $5L$,
$5M$, $5N$, $6m$, $7c$, $8d$, $9d$, $9e$, $10c$, $10d$, $10e$,
$10f$, and the nonminimal examples seen when considering small
spectral radius, $10g$, $10h$, $11a$, $11b$, $11c$, $12b$, $13a$,
$14a$, we produce Table \ref{Ta-2}, and establish Theorem
\ref{T:measure}. All the new examples include at least one charged
vertex.

\begin{picture}(300,90)
\put(0,75){Connected $5$- and $6$-vertex nonminimal noncyclotomic
charged signed} \put(-22,62){graphs with Mahler measure $<1.3$:}

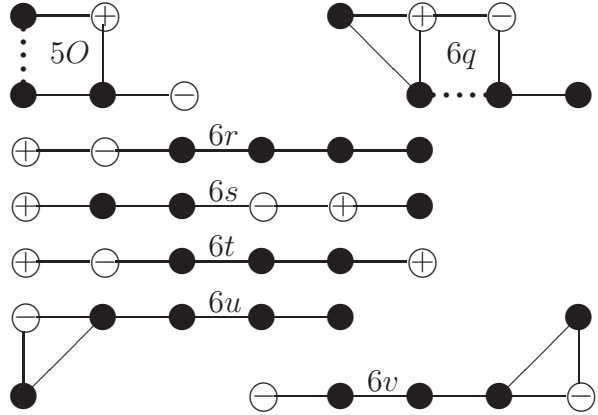
\multiput(60,0)(30,0){2}{\circle*{10}}
\put(115,-3.5){\textcircled{$-$}}
\multiput(210,0)(30,0){3}{\circle*{10}}
\multiput(60,30)(120,0){2}{\circle*{10}}
\multiput(85,26.5)(120,0){2}{\textcircled{+}}
\put(235,26.5){\textcircled{$-$}} \put(70,12){$5O$}
\put(220,12){$6q$} \put(60,0){\line(1,0){55}}
\multiput(210,0)(5,0){6}{\circle*{2}} \put(240,0){\line(1,0){30}}
\multiput(60,0)(0,5){6}{\circle*{2}} \put(90,0){\line(0,1){26.5}}
\multiput(210,0)(30,0){2}{\line(0,1){24.5}}
\put(60,30){\line(1,0){25}} \put(180,30){\line(1,0){25}}
\put(215,30){\line(1,0){20}} \put(180,30){\line(1,-1){30}}

\end{picture}

\begin{picture}(300,20)

\put(55,-3.5){\textcircled{+}} \put(85,-3.5){\textcircled{$-$}}
\multiput(120,0)(30,0){4}{\circle*{10}} \put(130,2){$6r$}
\put(65,0){\line(1,0){20}} \put(95,0){\line(1,0){115}}

\end{picture}

\begin{picture}(300,20)
\multiput(55,-3.5)(120,0){2}{\textcircled{+}}
\multiput(90,0)(30,0){2}{\circle*{10}}
\put(145,-3.5){\textcircled{$-$}} \put(210,0){\circle*{10}}
\put(130,2){$6s$} \put(65,0){\line(1,0){80}}
\put(155,0){\line(1,0){20}} \put(185,0){\line(1,0){25}}

\end{picture}

\begin{picture}(300,20)

\multiput(55,-3.5)(150,0){2}{\textcircled{+}}
\put(85,-3.5){\textcircled{$-$}}
\multiput(120,0)(30,0){3}{\circle*{10}} \put(130,2){$6t$}
\put(65,0){\line(1,0){20}} \put(95,0){\line(1,0){110}}

\end{picture}

\begin{picture}(300,50)

\put(60,0){\circle*{10}} \put(145,-3.5){\textcircled{$-$}}
\put(265,-3.5){\textcircled{$-$}}
\multiput(180,0)(30,0){3}{\circle*{10}}
\put(55,26.5){\textcircled{$-$}}
\multiput(90,30)(30,0){4}{\circle*{10}} \put(270,30){\circle*{10}}
\put(130,32){$6u$} \put(190,2){$6v$} \put(60,0){\line(0,1){24.5}}
\put(60,0){\line(1,1){30}} \put(155,0){\line(1,0){110}}
\put(240,0){\line(1,1){30}} \put(270,5){\line(0,1){25}}
\put(65,30){\line(1,0){115}}

\end{picture}

\begin{picture}(300,105)
\put(0,75){Connected $7$-vertex nonminimal noncyclotomic charged
signed graphs} \put(90,62){with Mahler measure $<1.3$:}

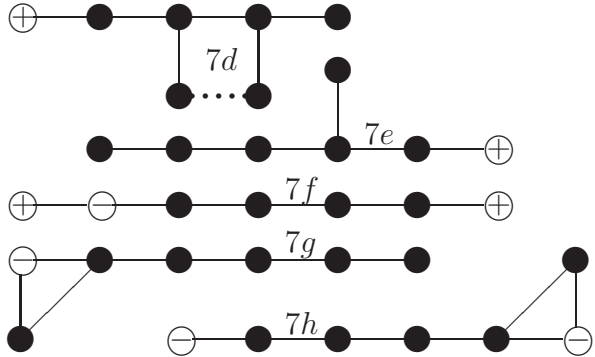
\multiput(100,20)(30,0){2}{\circle*{10}}
\put(35,46.5){\textcircled{+}}
\multiput(70,50)(30,0){4}{\circle*{10}} \put(110,30){$7d$}
\multiput(100,20)(5,0){6}{\circle*{2}}
\multiput(100,20)(30,0){2}{\line(0,1){30}}
\put(45,50){\line(1,0){115}}

\multiput(70,0)(30,0){5}{\circle*{10}}
\put(215,-3.5){\textcircled{+}} \put(160,30){\circle*{10}}
\put(170,2){$7e$} \put(70,0){\line(1,0){145}}
\put(160,0){\line(0,1){30}}
\end{picture}

\begin{picture}(300,20)
\multiput(35,-3.5)(180,0){2}{\textcircled{+}}
\put(65,-3.5){\textcircled{$-$}}
\multiput(100,0)(30,0){4}{\circle*{10}} \put(140,2){$7f$}
\put(45,0){\line(1,0){20}} \put(75,0){\line(1,0){140}}
\end{picture}

\begin{picture}(300,50)
\put(40,0){\circle*{10}} \put(35,26.5){\textcircled{$-$}}
\multiput(70,30)(30,0){5}{\circle*{10}} \put(140,33){$7g$}
\put(40,0){\line(0,1){24.5}} \put(40,0){\line(1,1){30}}
\put(45,30){\line(1,0){145}}

\put(95,-3.5){\textcircled{$-$}}
\multiput(130,0)(30,0){4}{\circle*{10}}
\put(245,-3.5){\textcircled{$-$}} \put(250,30){\circle*{10}}
\put(140,2){$7h$} \put(105,0){\line(1,0){140}}
\put(220,0){\line(1,1){30}} \put(250,4.5){\line(0,1){25}}

\end{picture}

\begin{picture}(300,85)
\put(0,55){Connected $8$-vertex nonminimal noncyclotomic charged
signed graphs} \put(-22,42){with Mahler measure $<1.3$:}

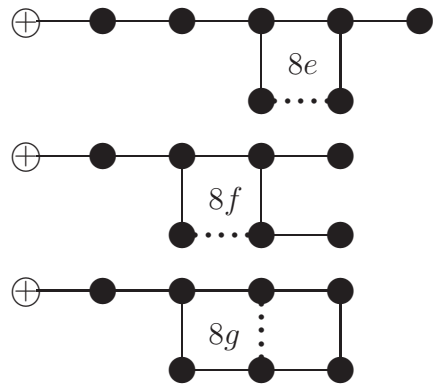
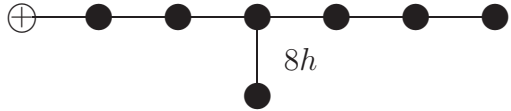
\multiput(150,0)(30,0){2}{\circle*{10}}
\put(55,26.5){\textcircled{+}}
\multiput(90,30)(30,0){5}{\circle*{10}} \put(160,10){$8e$}
\multiput(150,0)(5,0){6}{\circle*{2}}
\multiput(150,0)(30,0){2}{\line(0,1){30}}
\put(65,30){\line(1,0){145}}
\end{picture}

\begin{picture}(300,50)

\multiput(120,0)(30,0){3}{\circle*{10}}
\put(55,26.5){\textcircled{+}}
\multiput(90,30)(30,0){4}{\circle*{10}} \put(130,10){$8f$}
\multiput(120,0)(5,0){6}{\circle*{2}} \put(150,0){\line(1,0){30}}
\multiput(120,0)(30,0){2}{\line(0,1){30}}
\put(65,30){\line(1,0){115}}
\end{picture}

\begin{picture}(300,50)
\multiput(120,0)(30,0){3}{\circle*{10}}
\put(55,26.5){\textcircled{+}}
\multiput(90,30)(30,0){4}{\circle*{10}} \put(130,10){$8g$}
\put(120,0){\line(1,0){60}}
\multiput(120,0)(60,0){2}{\line(0,1){30}}
\multiput(150,0)(0,5){6}{\circle*{2}} \put(65,30){\line(1,0){115}}
\end{picture}

\begin{picture}(300,50)
\put(150,0){\circle*{10}} \put(55,26.5){\textcircled{+}}
\multiput(90,30)(30,0){6}{\circle*{10}} \put(160,10){$8h$}
\put(150,0){\line(0,1){30}} \put(65,30){\line(1,0){175}}
\end{picture}

\begin{picture}(300,85)
\put(0,55){Connected $9$-vertex nonminimal noncyclotomic charged
signed graphs} \put(-22,42){with Mahler measure $<1.3$:}

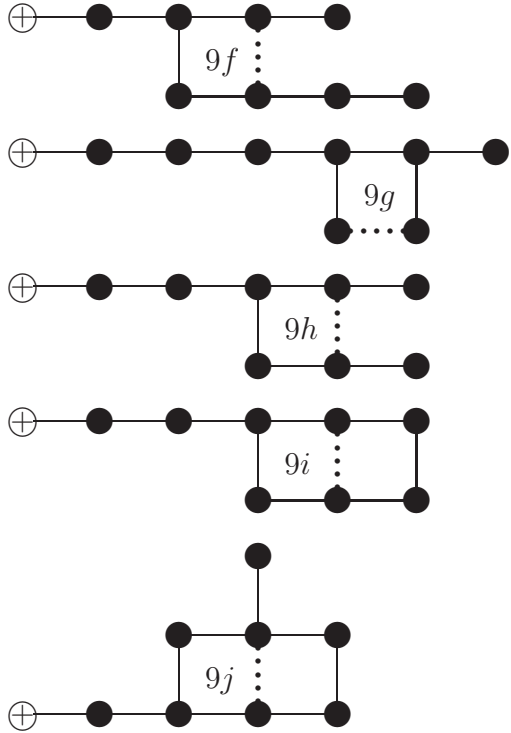
\multiput(90,0)(30,0){4}{\circle*{10}}
\put(25,26.5){\textcircled{+}}
\multiput(60,30)(30,0){4}{\circle*{10}} \put(100,10){$9f$}
\put(90,0){\line(1,0){90}} \put(90,0){\line(0,1){30}}
\multiput(120,0)(0,5){6}{\circle*{2}} \put(35,30){\line(1,0){115}}
\end{picture}

\begin{picture}(300,50)
\multiput(150,0)(30,0){2}{\circle*{10}}
\put(25,26.5){\textcircled{+}}
\multiput(60,30)(30,0){6}{\circle*{10}} \put(160,10){$9g$}
\multiput(150,0)(5,0){6}{\circle*{2}}
\multiput(150,0)(30,0){2}{\line(0,1){30}}
\put(35,30){\line(1,0){175}}
\end{picture}

\begin{picture}(300,50)
\multiput(120,0)(30,0){3}{\circle*{10}}
\put(25,26.5){\textcircled{+}}
\multiput(60,30)(30,0){5}{\circle*{10}} \put(130,10){$9h$}
\put(120,0){\line(1,0){60}} \put(120,0){\line(0,1){30}}
\multiput(150,0)(0,5){6}{\circle*{2}} \put(35,30){\line(1,0){145}}
\end{picture}

\begin{picture}(300,50)
\multiput(120,0)(30,0){3}{\circle*{10}}
\put(25,26.5){\textcircled{+}}
\multiput(60,30)(30,0){5}{\circle*{10}} \put(130,10){$9i$}
\put(120,0){\line(1,0){60}}
\multiput(120,0)(60,0){2}{\line(0,1){30}}
\multiput(150,0)(0,5){6}{\circle*{2}} \put(35,30){\line(1,0){145}}
\end{picture}

\begin{picture}(300,80)
\put(25,-3.5){\textcircled{+}}
\multiput(60,0)(30,0){4}{\circle*{10}}
\multiput(90,30)(30,0){3}{\circle*{10}} \put(120,60){\circle*{10}}
\put(100,10){$9j$} \put(35,0){\line(1,0){115}}
\put(90,30){\line(1,0){60}}
\multiput(90,0)(60,0){2}{\line(0,1){30}}
\multiput(120,0)(0,5){6}{\circle*{2}} \put(120,30){\line(0,1){30}}
\end{picture}

\begin{picture}(300,145)(0,-10)
\put(0,105){Connected $10$- and $11$-vertex nonminimal
noncyclotomic charged signed graphs} \put(-22,92){with Mahler
measure $<1.3$ other than $10g$, $10h$, $11a$, $11b$, $11c$:}

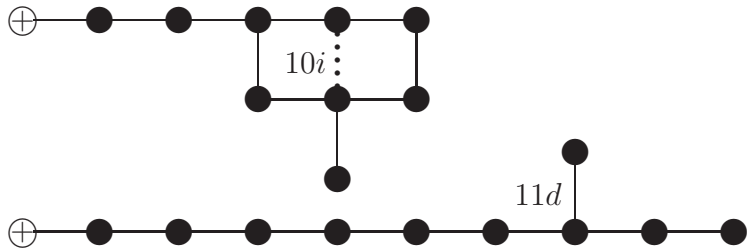
\put(25,76.5){\textcircled{+}}
\multiput(60,80)(30,0){5}{\circle*{10}}
\multiput(120,50)(30,0){3}{\circle*{10}}
\put(150,20){\circle*{10}} \put(130,60){$10i$}
\put(120,50){\line(1,0){60}} \put(35,80){\line(1,0){145}}
\multiput(120,50)(60,0){2}{\line(0,1){0}}
\multiput(150,50)(0,5){6}{\circle*{2}}
\put(150,20){\line(0,1){30}}

\multiput(120,50)(60,0){2}{\line(0,1){30}}

\put(25,-3.5){\textcircled{+}}
\multiput(60,0)(30,0){9}{\circle*{10}} \put(240,30){\circle*{10}}
\put(217,10){$11d$} \put(35,0){\line(1,0){265}}
\put(240,0){\line(0,1){30}}
\end{picture}

{\bf Acknowledgement.} We thank Klas Markstr\"om for pointing us
to the reference \cite{HR}.

\

\


\begin{thebibliography}{999}

\bibitem{Bh96} R.~Bhatia, Matrix analysis.  New York, NY: Springer 1996.

\bibitem{BDM} P.~Borwein, E.~Dobrowolski, M.~Mossinghoff, Lehmer's
problem for polynomials with odd coefficients.  {\em Ann.~of Math.}
{\bf 166} (2007), 347--366.

\bibitem{Bo77} D.W.~Boyd, Small Salem numbers.  {\em Duke Math. J.}  \textbf{44}  (1977),  315--328.

\bibitem{BN}A.E.~ Brouwer and A.~Neumaier, The graphs with spectral radius between
$2$ and $\sqrt{2+\sqrt{5}}$. {\em Linear Algebra Appl.} \textbf{114/115}
(1989), 273--276.

\bibitem{Ca} A.L. Cauchy, Sur l'\'equation \`a l'aide de laquelle on d\'etermine les in\'egalit\'es
s\'eculaires des mouvements des plan\`etes, Oeuvres compl\`etes,
IIi\`eme S\'erie, {\bf 9}, Gauthier-Villars, 174--195 (1829).


\bibitem{Cetal} P.J.~Cameron, J.J.~Seidel and S.V.~Tsaranov,
Signed graphs, root lattices, and Coxeter groups,
\textit{J.~Algebra} \textbf{164} (1994), 173--209.

\bibitem{CDG} D.~Cvetkovi\'{c}, M.~Doob and I.~Gutman,
On graphs whose eigenvalues do not exceed $\sqrt{2+\sqrt{5}}$,
\textit{Ars Combinatoria} \textbf{14} (1982), 225--239.

\bibitem{CR} D.~Cvetkovi\'c and P.~ Rowlinson, The largest eigenvalue of a graph: a survey. {\em  Linear and Multilinear Algebra }\textbf{28} (1990), 3--33.


\bibitem{D} E.~Dobrowolski,
A note on integer symmetric matrices and Mahler's measure, {\em
Canad. Math. Bull.}    \textbf{51} (2008), 57--59.

\bibitem{EG} D.R.~Estes and R.M.~Guralnick, Minimal polynomials of
integral symmetric matrices, \textit{Linear Algebra Appl.}
\textbf{192} (1993), 83--99.

\bibitem{F} S.~Fisk, A very short proof of Cauchy's interlace
theorem for eigenvalues of Hermitian matrices,  {\em Amer. Math.
Monthly} {\bf 112} (2005), 118. See also {\tt
arXiv:math.CA/0502408v1}.


\bibitem{HR} R.~Hayward,  B.A.~Reed,  Forbidding holes and antiholes. {\it in}
 J.L.~Ramírez Alfonsín, (ed.) et al., Perfect
graphs. Chichester: Wiley.  113-137 (2001).



\bibitem{Kro} L.~Kronecker, Zwei s\"atse \"uber gleichungen mit
ganzzahligen coefficienten, \emph{J.~Reine Angew.~Math.}
\textbf{53} (1857), 173--175.

\bibitem{MS} J.F.~McKee and C.J.~Smyth,
Salem numbers, Pisot numbers, Mahler measure, and graphs,
\textit{Experimental Mathematics} \textbf{14} (2005),
211--229.

\bibitem{MScyc} J.F.~McKee and C.J.~Smyth, Integer symmetric
matrices having all their eigenvalues in the interval $[-2,2]$,
\textit{J.~Algebra} \textbf{317} (2007), 260--290.

\bibitem{Moss} M.~Mossinghoff, List of small Salem numbers,\newline
{\tt http://www.cecm.sfu.ca/$\sim$mjm/Lehmer/lists/SalemList.html}



\bibitem{She} J.B.~Shearer, On the distribution of the maximum eigenvalue of graphs,  \emph{Linear Algebra Appl.}  \textbf{114/115}  (1989), 17--20.




\bibitem{S} C.J.~Smyth, The Mahler measure of algebraic numbers: a survey. \textit{Number theory and polynomials}, pp. 322--349.
(Conference proceedings, University of Bristol, 3-7 April 2006,
editors James McKee and Chris Smyth). LMS Lecture Note Series
\textbf{352}, Cambridge University Press, Cambridge, 2008. See
also {\tt arXiv:math.NT/0701397v2}.


\end{thebibliography}
\end{document}